\documentclass[a4paper,11pt,final]{article}

\usepackage{amssymb,amsthm,amsmath}
\usepackage{enumitem}
\usepackage{a4wide}
\usepackage{hyperref}
\usepackage[utf8]{inputenc}
\usepackage{graphicx}

\numberwithin{equation}{section}
\allowdisplaybreaks[4]

\theoremstyle{plain}
\newtheorem{proposition}{Proposition}[section]
\newtheorem{theorem}[proposition]{Theorem}
\newtheorem{lemma}[proposition]{Lemma}
\newtheorem{corollary}[proposition]{Corollary}

\theoremstyle{definition}
\newtheorem{definition}[proposition]{Definition}

\theoremstyle{remark}

\renewenvironment{proof}{\smallskip\noindent\emph{\textbf{Proof.}}\hspace{1pt}}%
{\hspace{-5pt}{\nobreak\quad\nobreak\hfill\nobreak$\square$\vspace{8pt}%
\par}\smallskip\goodbreak}

\newenvironment{proofof}[1]{\smallskip\noindent\emph{\textbf{Proof of #1.}}%
\hspace{1pt}}{\hspace{-5pt}{\nobreak\quad\nobreak\hfill\nobreak%
$\square$\vspace{8pt}\par}\smallskip\goodbreak}

\renewcommand{\leq}{\leqslant}
\renewcommand{\geq}{\geqslant}

\newcommand{\pint}[1]{\mathaccent23{#1}}
\newcommand{\C}[1]{\mathbf{C}^{#1}}
\newcommand{\Cc}[1]{\mathbf{C}_c^{#1}}
\newcommand{\modulo}[1]{{\left|#1\right|}}
\newcommand{\norma}[1]{{\left\|#1\right\|}}
\newcommand{\reali}{{\mathbb{R}}}
\newcommand{\naturali}{{\mathbb{N}}}

\newcommand{\BV}{\mathbf{BV}}

\renewcommand{\epsilon}{\varepsilon}
\renewcommand{\phi}{\varphi}
\renewcommand{\theta}{\vartheta}
\renewcommand{\L}[1]{{\mathbf{L}^#1}}

\newcommand{\W}[2]{{\mathbf{W}^{#1,#2}}}
\renewcommand{\O}{\mathcal{O}(1)}
\newcommand{\Lloc}[1]{{\mathbf{L}_{loc}^{#1}}}
\newcommand{\tv}{\mathinner{\rm TV}}

\renewcommand{\div}{\nabla_x\cdot}%{\mathinner{\rm div}}

\newcommand{\spt}{\mathop{\rm spt}}
\newcommand{\sgn}{\mathop{\rm sgn}}

\renewcommand{\d}[1]{\mathinner{\mathrm{d}{#1}}}

\begin{document}

\title{Differential Equations Modeling Crowd Interactions}

\author{Raul Borsche$^1$, Rinaldo M.~Colombo$^2$, Mauro Garavello$^3$,
  Anne Meurer$^1$}

\footnotetext[1]{Fachbereich Mathematik, Technische Universit\"at
  Kaiserslautern}

\footnotetext[2]{Unit\`a INdAM, c/o DII, Universit\`a degli studi di
  Brescia}%.  \texttt{rinaldo.colombo@unibs.it} }

\footnotetext[3]{Dipartimento di Matematica e Applicazioni,
  Universit\`a di Milano
  Bicocca}%. \texttt{mauro.garavello@unimib.it}}

\maketitle

\begin{abstract}

  \noindent Nonlocal conservation laws are used to describe various
  realistic instances of crowd behaviors. First, a basic analytic
  framework is established through an \emph{ad hoc} well posedness
  theorem for systems of nonlocal conservation laws in several space
  dimensions interacting non locally with a system of ODEs.  Numerical
  integrations show possible applications to the interaction of
  different groups of pedestrians, and also with other \emph{agents}.

  \medskip

  \noindent\textit{2000~Mathematics Subject Classification:} 35L65,
  90B20.

  \medskip

\end{abstract}

\section{Introduction}
\label{sec:I}

This paper deals with a system composed by several populations and
individuals, or agents. The former are described through their
\emph{macroscopic densities}, the latter through \emph{discrete
  points}. In analytic terms, this leads to a system of conservation
laws coupled with ordinary differential equations. From a modeling
point of view, it is natural to encompass also interactions that are
\emph{nonlocal}, in both cases of interactions within the populations
as well as between each population and each individual agent.

Throughout, $t \in \reali^+$ is time and the space coordinate is $x
\in \reali^d$. The number of populations is $n$ and their densities
are $\rho^i = \rho^i (t,x)$, for $i=1, \ldots, n$. The individuals are
described through a vector $p = p (t)$, with $p \in \reali^m$. In the
case of $N$ agents, $p$ may consist of the vector of each individual
position, so that $m = N\,d$, or else it may contain also each
individual speed, so that $m = 2 \, N \, d$.

Setting $\rho = (\rho^1, \ldots, \rho^n)$, we are thus lead to
consider the system
\begin{equation}
  \label{eq:1}
  \left\{
    \begin{array}{l}
      \displaystyle
      \partial_t \rho^i
      +
      \div
      \left[
        q^i(\rho^i) \;
        v^i \!\! \left(
          t, x, \left(\mathcal{A}^i \left(\rho (t)\right)\right) (x), p
        \right)
      \right]
      =
      0,
      \\[10pt]
      \displaystyle
      \dot p = F\left(t,p,\left(\mathcal{B} \left(\rho (t)\right)\right) (p)\right),
    \end{array}
  \right.
\end{equation}
where $\mathcal{A}^i$ and $\mathcal{B}$ are nonlocal operators,
reflecting the fact that the behavior of the members of the population
as well as of the agents depends on suitable spatial averages. The
function $v^i$ gives the speed of the $i$-th population, and $F$ yields
the evolution of the individuals. We defer to Section~\ref{sec:AR} for
the precise definitions and regularity requirements.

Motivations for the study of~\eqref{eq:1} are found, for instance,
in~\cite{BCG_Nonlinearity, BCG_JDE, m3as-cgm, ColomboMercier,
  ColomboMagali}, which all provide examples of realistic situations
that fall within~\eqref{eq:1}. Beside these, system~\eqref{eq:1} also
allows to describe new scenarios, some examples are considered in
detail in Section~\ref{sec:EX}. There, we limit our scope to
$\reali^2$ (i.e., $d=2$) essentially due to visualization problems in
higher dimensions. The analytic treatment below, however, is fully
established in any spacial dimension.

As a first example, in Section~\ref{sub:tourist} we study two groups
of tourists each following a guide. The two groups are described
through the pedestrian model in~\cite{m3as-cgm, ColomboMercier,
  ColomboMagali} and the guides move according to an ODE. Each group follows its
guide and interacts with the other group, while both guides need to
wait for their respective group.

Section~\ref{sub:car} is devoted to pedestrians crossing a
street at a crosswalk, while cars are driving on the road. The
pedestrians' movement is described as in the previous example, the
attractive role of the guides being substituted by a repulsive effect
of cars on pedestrians. On the other hand, cars move according to a
follow the leader model and try to avoid hitting pedestrians. This
results in a strong coupling between the ODE and PDE, since the
pedestrians can not cross the street if a car is coming and on the
other hand the cars have to stop if there are people on the road.

As a third example, see Section~\ref{sub:hools}, two groups of
hooligans confront with each other. Police officers try to separate
the two groups heading towards the areas with the strongest mixing of
hooligans. Thus, they move according to the densities of the
hooligans, which themselves try to avoid the contact with the police.
All examples are illustrated by numerical integrations showing central
features of the models.

The current literature offers alternative approaches to the modeling
of crowds \cite{helbing1995social,hughes2002continuum}. Notably, we recall the so called \emph{multiscale}
framework, based on measure valued differential equations,
see~\cite{CristianiPiccoliTosin, PiccoliTosin2009,
  PiccoliTosin}. There, the interplay between the atomic part and the
absolutely continuous part of the unknown measure reminds of the
present interplay between the PDE and the ODE. Nevertheless,
differently from the cited references, here we exploit the distinct
nature of the two equations to assign different roles to agents and
crowds.

This paper is organized as follows: in Section~\ref{sec:AR} we give a
precise definition of a solution of system~\eqref{eq:1} and state the
main analytic results.  In Section~\ref{sec:EX} we describe three
examples which fit into the above framework and present accompanying numerical
integrations.  All the technical details are collected in
Section~\ref{sec:technical}.

\section{Analytical Results}
\label{sec:AR}
In this section we state some analytical results for solutions of~\eqref{eq:1}.
Throughout we denote $\reali^+ = [0, +\infty[$, $R$ is a positive
constant and $I \subseteq \reali^+$ is an interval containing $0$.

The function $q^i$ describes the internal dynamics of the population $\rho^i$ and is required to satisfy
\begin{description}
\item[(q)] $q^i \in \C2\left(\reali^+; \reali^+\right)$ satisfies $q^i
  (0) = 0$ and $q^i(R) = 0$.
\end{description}
For the ``velocity'' vectors $v^i$ we require the following
regularity
\begin{description}
\item[(v)] For every $i \in \left\{1, \ldots, n\right\}$ the velocity
  $v^i: \reali^+ \times \reali^d \times \reali^{d} \times \reali^m \to
  \reali^d$ is such that
  \begin{description}
  \item[(v.1)] $v^i \in (\C2 \cap \L\infty) (\reali^+ \times \reali^d \times
    \reali^{d} \times \reali^m; \reali^d)$.
  \item[(v.2)] For all $T \in \reali^+$ and all compact set $K \subseteq
    \reali^m$, there exists a function $\mathcal{C}_K \in (\L1 \cap
    \L\infty) (\reali^d; \reali^+)$ such that, for $t \in [0,T]$, $x
    \in \reali^d$, $A \in \reali^{d}$ and $p \in K$
    \begin{equation*}
      \begin{array}{@{}r@{\,}c@{\,}l@{\quad}r@{\,}c@{\,}l@{}}
        \norma{v^i\left(t,x, A,p\right)}_{\reali^d}
        & < &
        \mathcal{C}_K (x),
        \\[.2cm]
        \norma{\nabla_x \cdot v^i\left(t,x, A,p\right)}_\reali
        & < &
        \mathcal{C}_K (x),
        &
        \norma{\nabla_x \nabla_x \cdot v^i\left(t,x,A,p\right)}_{\reali^d }
        & < &
        \mathcal{C}_K (x),
        \\[.2cm]
        \norma{\nabla_A v^i\left(t,x,A,p\right)}
        _{\reali^d \times \reali^{d}}
        & < &
        \mathcal{C}_K (x),
        &
        \norma{\nabla_p v^i\left(t,x,A,p\right)}_{\reali^{m} \times
          \reali^d}
        & < &
        \mathcal{C}_K (x),
        \\[.2cm]
        \norma{\nabla_x \nabla_A v^i\left(t,x,A,p\right)}
        _{\reali^d \times \reali^{d}\times \reali^{d}}
        & < &
        \mathcal{C}_K (x),
        &
        \norma{\nabla_x \nabla_p v^i\left(t,x,A,p\right)}_{\reali^m
          \times \reali^ d \times \reali^d }
        & < &
        \mathcal{C}_K (x),
        \\[.2cm]
        \norma{\nabla^2_A v^i\left(t,x,A,p\right)}
        _{\reali^d \times \reali^{d}\times \reali^{d}}
        & < &
        \mathcal{C}_K (x),
        &
        \norma{\nabla_A \nabla_p v^i\left(t,x,A,p\right)}_{\reali^m
          \times \reali^ d \times \reali^d }
        & < &
        \mathcal{C}_K (x).
      \end{array}
    \end{equation*}
  \end{description}
\end{description}

\noindent Remark however that~\textbf{(v.2)} becomes redundant as soon
as the initial datum to~\eqref{eq:1} has compact support and the
solution is seeked on a bounded time interval, see
Corollary~\ref{cor:zero}.

We impose to the ordinary differential equation in~\eqref{eq:1} to fit
into the usual framework of Caratheodory equations,
see~\cite[\S~1]{Filippov}, introducing the following conditions.
\begin{description}
  \setlength{\itemsep}{0pt} \setlength{\parskip}{0pt}
\item[(F)] The map $F \colon \reali^+ \times \reali^m \times
  \reali^\ell \longrightarrow \reali^m$ is such that
  \begin{enumerate}[leftmargin=0pt]
  \item For all $t>0$ and $b \in \reali^\ell$, the function $
    \begin{array}{ccc}
      \reali^m & \longrightarrow & \reali^m \\
      p & \longmapsto & F(t,p,b)
    \end{array}
    $ is continuous.

  \item For all $t>0$ and $p \in \reali^m$, the function $
    \begin{array}{ccc}
      \reali^\ell & \longrightarrow & \reali^m \\
      b & \longmapsto & F(t,p,b)
    \end{array}
    $ is continuous.

  \item For all $b \in \reali^\ell$ and $p \in \reali^m$, the function
    $
    \begin{array}{ccc}
      \reali^+ & \longrightarrow & \reali^m \\
      t & \longmapsto & F(t,p,b)
    \end{array}
    $ is Lebesgue measurable.

  \item For all compact subset $K$ of $\reali^m$, there exists a
    constant $L_F > 0$ such that, for every $t \in \reali^+$, $p_1,
    p_2 \in K$ and $b_1,b_2 \in \reali^\ell$,
    \begin{displaymath}
      \norma{F(t,p_1,b_1) - F(t,p_2,b_2)}_{\reali^m}
      \leq
      L_F \, \left(
        \norma{p_1 - p_2}_{\reali^m}
        +
        \norma{b_1 - b_2}_{\reali^\ell}
      \right) \,.
    \end{displaymath}

  \item There exists a function $C_F \in \Lloc1(\reali^+; \reali^+)$
    such that for all $t>0$, $b \in \reali^\ell$ and $p \in \reali^m$
    \begin{displaymath}
      \norma{F(t,p,b)}_{\reali^m}
      \leq
      C_F(t) \,
      \left(1 + {\norma{p}}_{\reali^m} + \norma{b}_{\reali^\ell}\right) \,.
    \end{displaymath}
  \end{enumerate}
  \end{description}
On the nonlocal operators $\mathcal{A,B}$ we require
\begin{description}
\item[($\boldsymbol{\mathcal{A}}$)] The maps $\mathcal{A}^i\colon
  \L1(\reali^d; \reali^n) \to (\C2 \cap \W21) (\reali^d; \reali^d)$
  are Lipschitz continuous and satisfy $\mathcal{A}^i (0) = 0$.  In
  particular there exists a positive constant $L_A>0$ such that, for
  every $\rho_1, \rho_2 \in \L1(\reali^d; [0,R]^n)$,
  \begin{equation*}
    \norma{\mathcal{A}^i (\rho_1) - \mathcal{A}^i (\rho_2)}_{\W21} +
    \norma{\mathcal{A}^i (\rho_1) - \mathcal{A}^i (\rho_2)}_{\C2}
    \leq L_A \norma{\rho_1 - \rho_2}_{\L1}.
  \end{equation*}

\item[($\boldsymbol{\mathcal{B}}$)] The map $\mathcal{B}\colon
  \L1(\reali^d; \reali^n) \to \W1{\infty}(\reali^m; \reali^\ell)$ is
  Lipschitz continuous and satisfies $\mathcal B(0) = 0$. In
  particular, there exists a positive constant $L_B > 0$ such that,
  for every $\rho_1, \rho_2 \in \L1(\reali^d; [0,R]^n)$,
  \begin{equation*}
    \norma{\mathcal{B} (\rho_1) - \mathcal{B} (\rho_2)}_{\W1{\infty}}
    \leq L_B \norma{\rho_1 - \rho_2}_{\L1}.
  \end{equation*}

\end{description}

\begin{definition}
  \label{def:solution-1}
  Fix $\rho_o \in (\L1 \cap \BV) (\reali^d; \reali^n)$ and $p_{o} \in
  \reali^m$.  A vector $\left(\rho, p\right)$ with
  \begin{displaymath}
    \rho \in \C0\left(I; \L1(\reali^d; \reali^n)\right)
    \qquad \textrm{ and } \qquad
    p \in \W11 (I; \reali^m)
  \end{displaymath}
  is a solution to~\eqref{eq:1} with $\rho (0,x) = \rho_o(x)$ and
  $p(0) = p_o$ if
  \begin{enumerate}
  \item For $i=1, \ldots, n$, the map $\rho^i$ is a Kru\v zkov
    solution to the scalar conservation law
    \begin{displaymath}
      \partial_t \rho^i
      +
      \div
      \left[
        q^i(\rho^i)\, V (t,x)
      \right]
      = 0
      \quad \mbox{ where } \quad
      V (t,x)
      =
      v^i
      \left(
        t, x,
        \left(\mathcal{A}^i \left(\rho (t)\right)\right) (x),
        p (t)
      \right).
    \end{displaymath}

  \item The map $p$ is a Caratheodory solution to the ordinary
    differential equation
    \begin{displaymath}
      \dot p = \mathcal{F} (t, p)
      \quad \mbox{ where } \quad
      \mathcal{F} (t,p)
      =
      F\left(t, p, \left(\mathcal{B} \left(\rho (t)\right)\right) (p)\right)\,.
    \end{displaymath}

  \item $\rho(0, x) = \rho_o(x)$ for a.e.~$x \in \reali^d$.

  \item $p(0) = p_{o}$.
  \end{enumerate}
\end{definition}

\noindent Above, for the definition of Kru\v zkov solution we refer
to~\cite[Definition~1]{Kruzkov}. By Caratheodory solution we mean the
solution to the integral equation,
see~\cite[Chapter~2]{bressan-piccoli-book}. Observe that
by~\textbf{(q)}, the function $(0,p)$, respectively $(R,p)$,
solves~\eqref{eq:1} as soon as $\dot p = F (t,p,0)$, respectively
$\dot p = F \left(t,p, \left(\mathcal{B} (R)\right) (p)\right)$.

We are now ready to state the main result of this work, whose proof is
deferred to Section~\ref{sec:technical}.

\begin{theorem}
  \label{thm:main}
  Assume that~\textbf{(v)}, \textbf{(F)},
  \textbf{($\boldsymbol{\mathcal{A}}$)},
  \textbf{($\boldsymbol{\mathcal{B}}$)} and~\textbf{(q)} hold.  Then,
  there exists a process
  \begin{displaymath}
    \mathcal{P} \colon
    \{(t_1,t_2) \colon t_2 \geq t_1 \geq 0\}
    \times
    (\L1 \cap \BV) (\reali^d; [0,R]^n) \times \reali^m
    \to
    (\L1 \cap \BV) (\reali^d; [0,R]^n) \times \reali^m
  \end{displaymath}
  such that
  \begin{enumerate}
  \item for all $t_1,t_2,t_3 \in\reali^+$ with $t_3 \geq t_2 \geq
    t_1$, $\mathcal{P}_{t_2,t_3} \circ \mathcal{P}_{t_1,t_2} =
    \mathcal{P}_{t_1,t_3}$ and $\mathcal{P}_{t,t}$ is the identity for
    all $t \in \reali^+$.
  \item For all $(\rho_o,p_o) \in (\L1 \cap \BV) (\reali^d; [0,R]^n)
    \times \reali^m$, the continuous map $t \mapsto
    \mathcal{P}_{t_o,t} (\rho_o, p_o)$, defined for $t \geq t_o$, is
    the unique solution to~\eqref{eq:1} in the sense of
    Definition~\ref{def:solution-1} with initial datum $(\rho_o, p_o)$
    assigned at time $t_o$.

  \item For any pair $(\rho_o^1,p_o^1), (\rho_o^2,p_o^2) \in (\L1\cap
    \BV) (\reali^d; [0,R]^n) \times \reali^m$, there exists a function
    $\mathcal{L} \in \C0 (\reali^+; \reali^+)$ such that $\mathcal{L}
    (0)=0$ and, setting $(\rho_i,p_i) (t) = \mathcal{P}_{0,t}
    (\rho_o^i,p_o^i)$,
    \begin{eqnarray*}
      \norma{\rho_1 (t) - \rho_2 (t)}_{\L1}
      & \leq &
      \left(1+\mathcal{L} (t)\right) \, \norma{\rho_o^1-\rho_o^2}_{\L1}
      +
      \mathcal{L} (t) \, \norma{p_o^1-p_o^2}_{\reali^m} \,,
      \\
      \norma{p_1 (t) - p_2 (t)}_{\reali^m}
      & \leq &
      \mathcal{L} (t) \, \norma{\rho_o^1-\rho_o^2}_{\L1}
      +
      \left(1+\mathcal{L} (t)\right)\,  \norma{p_o^1-p_o^2}_{\reali^m} \,.
    \end{eqnarray*}

  \item For any $(\rho_o, p_o) \in (\L1\cap \BV) (\reali^d; [0,R]^n)
    \times \reali^m$, if $q_1,q_2$, $v_1,v_2$ and $F_1,F_2$
    satisfy~\textbf{(q)}, \textbf{(v)} and~\textbf{(F)}, then there
    exists a function $\mathcal{K} \in \C0 (\reali^+; \reali^+)$ such
    that $\mathcal{K} (0)=0$ and, calling $(\rho_i, p_i)$ the
    corresponding solutions,
    \begin{eqnarray*}
      \norma{\rho_1 (t) - \rho_2 (t)}_{\L1}
      & \leq &
      \mathcal{K} (t)
      \left(
        \norma{q_1 - q_2}_{\W1\infty}
        +
        \norma{v_1 - v_2}_{\W1\infty}
        +
        \norma{F_1 - F_2}_{\L\infty}
      \right)\,,
      \\
      \norma{p_1 (t) - p_2 (t)}_{\reali^m}
      & \leq &
      \mathcal{K} (t)
      \left(
        \norma{q_1 - q_2}_{\W1\infty}
        +
        \norma{v_1 - v_2}_{\W1\infty}
        +
        \norma{F_1 - F_2}_{\L\infty}
      \right) \,.
    \end{eqnarray*}
  \end{enumerate}
\end{theorem}

\noindent As soon as the initial datum for~\eqref{eq:1} has compact
support, it is possible to avoid the requirement~\textbf{(v.2)} in the
assumptions of Theorem~\ref{thm:main}.

\begin{corollary}
  \label{cor:zero}
  Assume that~\textbf{(v.1)}, \textbf{(F)},
  \textbf{($\boldsymbol{\mathcal{A}}$)},
  \textbf{($\boldsymbol{\mathcal{B}}$)} and~\textbf{(q)} hold. For any
  positive $T$ and for any initial datum $(\rho_o, p_o) \in (\L1 \cap
  \BV) (\reali^d; [0,R]^n) \times \reali^m$ such that $\spt \rho_o$ is
  compact, there exists a function $\tilde v$ satisfying~\textbf{(v)}
  such that the solution $t \to \left(\rho (t), p (t)\right)$
  constructed in Theorem~\ref{thm:main} to
  \begin{equation}
    \label{eq:1tilde}
    \left\{
      \begin{array}{l}
        \displaystyle
        \partial_t \rho^i
        +
        \div
        \left[
          q^i(\rho^i) \;
          \tilde v^i \!\! \left(
            t, x, \left(\mathcal{A}^i \left(\rho (t)\right)\right) (x), p
          \right)
        \right]
        =
        0,
        \\[10pt]
        \displaystyle
        \dot p = F\left(t,p,\left(\mathcal{B} \left(\rho (t)\right)\right) (p)\right),
      \end{array}
    \right.
  \end{equation}
  with initial datum $(\rho_o, p_o)$ also solves~\eqref{eq:1} in the
  sense of Definition~\ref{def:solution-1} for $t \in
  [0,T]$. Moreover, $\spt\rho (t)$ is compact for all $t \in [0,T]$.
\end{corollary}

\noindent The detailed proof is in Section~\ref{sec:technical}.
%
%
% Examples
%
%
\section{Numerical Integrations}
\label{sec:EX}

In this section we present sample applications of system~\eqref{eq:1}
that fit into the framework of Theorem~\ref{thm:main} or
Corollary~\ref{cor:zero}. To show qualitative features of the
solutions, we numerically integrate~\eqref{eq:1}. More precisely, the
ODE is solved by means of the explicit forward Euler method, while for
the PDE we use a FORCE scheme on a triangular mesh,
see~\cite[\S~18.6]{Toro}. We use the same time step according to a CFL number 0.9 and to the
stability bound of the ODE solver.

The coupling is achieved by fractional stepping~\cite[\S~19.5]{LeVeque}.
All numerical integrations are based on the same framework.

\subsection{Guided Groups}
\label{sub:tourist}

We consider two groups of tourists following their own guide. Members
of both groups always aim to stay close to the respective guide, but
also try to avoid too crowded places. In this setting, we have $x \in
\reali^d$ with $d = 2$, $n = 2$ populations $\rho^i (t,x)$ describing
the density of the $i$-th group of tourists, $p = [p^1,p^2]^T \in
\reali^m$ with $m = 4$, where $p^i$ describes the position in
$\reali^2$ of the guide of the $i$-th group. The density $\rho^i$
solves the conservation law
\begin{equation}
  \label{eq:density-tourists}
  \partial_t \rho^i + \nabla_x \cdot \left[\rho^i\left(1 - \rho^i\right)
    \left(w^i(p^i-x) - {\mathcal A}^i(\rho) \right)\right] = 0
\end{equation}
as in~\cite{m3as-cgm, ColomboMagali}, where
\begin{equation}
  \label{eq:A_tourist}
  w^i(\xi) =    \epsilon_i \, \frac{\xi}{
    \sqrt{1 + \norma{\xi}_{\reali^2}^4}}
  \quad \mbox{ and }\quad
  \mathcal{A}^i(\rho)
  =
  \sum_{j=1}^2 \epsilon_{ij}
  \frac{\nabla_x (\rho^j\ast \eta)}{\sqrt{1+\norma{\nabla_x
        (\rho^j\ast\eta)}_{\reali^2}^2}} \,.
\end{equation}
Here, $\epsilon_i$ and $\epsilon_{ij}$ are positive constants and
$\eta \in \Cc2(\reali^2;\reali^+)$. Moreover, $w^i (p^i-x)$ describes
the interaction between the member at $x$ of the $i$-th population and
his/her guide at $p^i$. The $2$ addends in the non local operator
$\mathcal{A}^i$ model the interaction among members of the same
population, the $\epsilon_{ii}$ term, and between the two populations,
the $\epsilon_{ij}$ term.

The leaders $p^1$ and $p^2$ adapt their speed according to the amount
of members of their group nearby. We assume that $p^i$ is constrained
to the circumference of radius $r^i$, centered at the point $c^i =
[c^i_1, c^i_2]^T \in \reali^2$, and its speed depends on an average
density of tourists around its position. Hence, $\ell = 2$ and
\begin{equation}
  \label{eq:p-tourists}
  \left\{
    \begin{array}{rcr}
      \dot p^i_1(t) %= r^i \cos\left(\theta^i(t)\right) \dot \theta^i(t)
      & = &
      d^i \left(p^i_2(t) - c^i_2\right) (\bar \eta * \rho^i)
      \left(p^i(t)\right),
      \\
      \dot p^i_2(t) %= - r^i \sin\left(\theta^i(t)\right) \dot \theta^i(t)
      & = &
      - d^i \left(p^i_1(t) - c^i_1\right)     (\bar \eta * \rho^i)
      \left(p^i(t)\right),
    \end{array}
  \right.
  \qquad i=1,2 \,,
\end{equation}
where $d^i$ is a real parameter.
System~\eqref{eq:density-tourists}-\eqref{eq:p-tourists} fits
into~\eqref{eq:1} by setting
\begin{equation}
  \label{eq:group-func-def}
  \begin{array}{rcl@{\qquad}rcl}
    q^1(\rho^1) & = & \rho^1(1-\rho^1),
    &
    F_1 (t,p,B) & = & d^1 \, \left(p^1_2 - c^1_2\right) B_1,
    \\
    q^2(\rho^2) & = & \rho^2(1-\rho^2),
    &
    F_2 (t,p,B) & = & - d^1 \, \left(p^1_1 - c^1_1\right) B_1,
    \\
    v^1(t,x,A,p) & = & w^1(p^1-x) - A,
    &
    F_3 (t,p,B) & = & d^2 \, \left(p^2_2 - c^2_2\right) B_2,
    \\
    v^2(t,x,A,p) & = & w^2(p^2-x) - A,
    &
    F_4 (t,p,B) & = & - d^2 \, \left(p^2_1 - c^2_1\right) B_2,
    \\
    \mathcal{A}^1(\rho)
    & = &
    \displaystyle
    \sum_{j=1}^2
    \frac{\epsilon_{1j} \nabla_x (\rho^j\ast \eta)}{\sqrt{1+\norma{\nabla_x
          (\rho^j\ast\eta)}_{\reali^2}^2}},
    &
    \left({\mathcal B}(\rho)\right) (p)
    & = &
    \left([\rho^1 , \rho^2 ]^T \ast \bar \eta\right) (p),
    \\
    \mathcal{A}^2(\rho)
    & = &
    \displaystyle
    \sum_{j=1}^2
    \frac{\epsilon_{2j}\nabla_x (\rho^j\ast \eta)}{\sqrt{1+\norma{\nabla_x
          (\rho^j\ast\eta)}_{\reali^2}^2}},
  \end{array}
\end{equation}
where we write $F = [F_1, F_2, F_3, F_4]^T$.

\begin{proposition}
  \label{prop:Ex1}
  Assume $\eta, \bar \eta \in \Cc2(\reali^2; \reali^+)$. Then, the
  functions defined in~\eqref{eq:group-func-def}
  satisfy~\textbf{(v.1)}, \textbf{(F)},
  \textbf{($\boldsymbol{\mathcal{A}}$)},
  \textbf{($\boldsymbol{\mathcal{B}}$)} and~\textbf{(q)}. In
  particular, Corollary~\ref{cor:zero} applies
  to~\eqref{eq:density-tourists}-\eqref{eq:p-tourists}-\eqref{eq:group-func-def}.
\end{proposition}

\noindent The proof is deferred to Section~\ref{sec:technical}.

As a specific example we consider the situation identified by the
following parameters
\begin{equation*}
  \begin{array}{rcl@{\qquad}rcl@{\qquad}rcl@{\qquad}rcl}
    \epsilon_1 & = & 0.4, &
    \epsilon_2 & = & 0.4, &
    \epsilon_{11} & = & 0.2, &
    \epsilon_{22} & = & 0.2,
    \\
    \epsilon_{12} & = & 0.8, &
    \epsilon_{21} & = & 0.8, &
    c^1 & = & [2,2]^T, &
    c^2 & = & [2,3]^T,
    \\
    r^1 & = & 1, &
    r^2 & = & 1, &
    d^1 & = & 1, &
    d^2 & = & -1,
  \end{array}
\end{equation*}
and by the functions
\begin{displaymath}
  \begin{array}{@{}r@{\;}c@{\;}l@{\;\mbox{ where }\;}r@{\;}c@{\;}l@{}}
    \eta(x)
    & = &
    \displaystyle
    \frac{\tilde \eta_1(x)}{\int_{\reali^2} \tilde\eta_1(x) \d x},
    &
    \tilde \eta_1 (x)
    & = &
    \left\{
      \begin{array}{ll}
        \exp \left(
          - \frac{20 x_1^2}{1 - 4 x_1^2}
          - \frac{20 x_2^2}{1 - 4 x_2^2}
        \right),
        & x \in [-0.5, \, 0.5]^2,
        \\
        0, & \mbox{otherwise},
      \end{array}
    \right.
    \\
    \bar \eta (x)
    & = &
    \displaystyle
    \frac{\tilde \eta_2(x)}{\int_{\reali^2} \tilde\eta_2(x) \d x},
    &
    \tilde \eta_2 (x)
    & = &
    \left\{
      \begin{array}{ll}
        \left(1 - \left(\frac{5x_1}{2}\right)^2\right)^3
        \left(1 - \left(\frac{5x_2}{2}\right)^2\right)^3,
        & x \in [-0.4, \, 0.4]^2,
        \\
        0, & \mbox{otherwise.}
      \end{array}
    \right.
  \end{array}
\end{displaymath}
The computational domain is $[0,1]^2$ and as initial conditions we
choose
\begin{align*}
  \rho_o^1 = 0.75\, \chi_{\strut[0.5,1.5]\times[0.5,1.5]} \qquad
  \rho_o^2 = \chi_{\strut [2.5,3.5]\times[0.5,1.5]} \qquad p_o^1
  = [2,3]^T, \qquad p_o^2 = [2,2]^T.
\end{align*}
In Figure~\ref{fig:po_int_tourist}, the solution up to $T\approx40$ is
shown.  The densities of the groups are the blue (for $i=1$) and red
(for $i=2$) regions, whereas the guides are located at the dots of the
corresponding color.
\begin{figure}[ht!]
  \centering
  \includegraphics[width=0.45\textwidth, viewport=150 50 1350
  800]{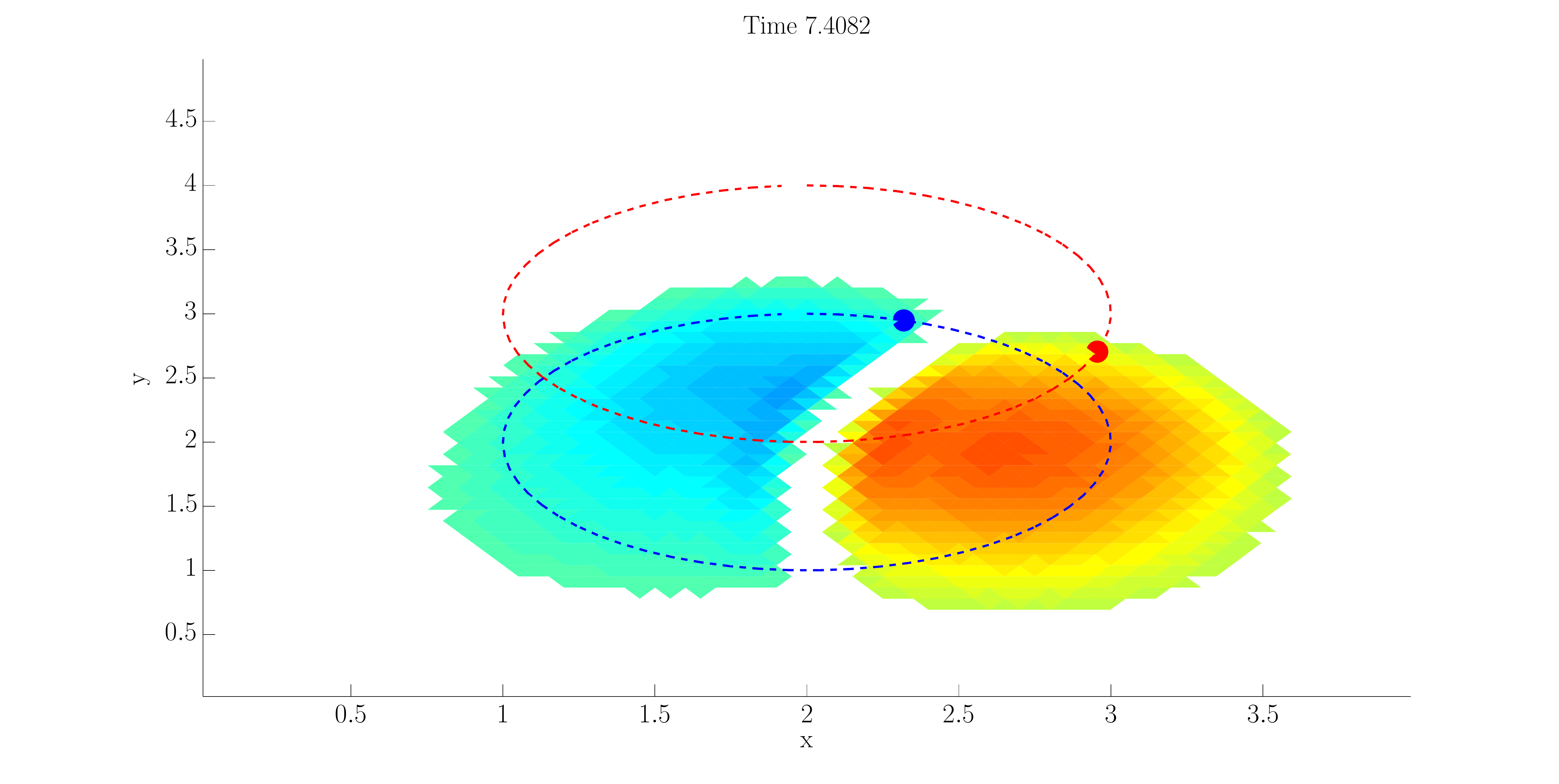}%
  \includegraphics[width=0.45\textwidth, viewport=150 50 1350
  800]{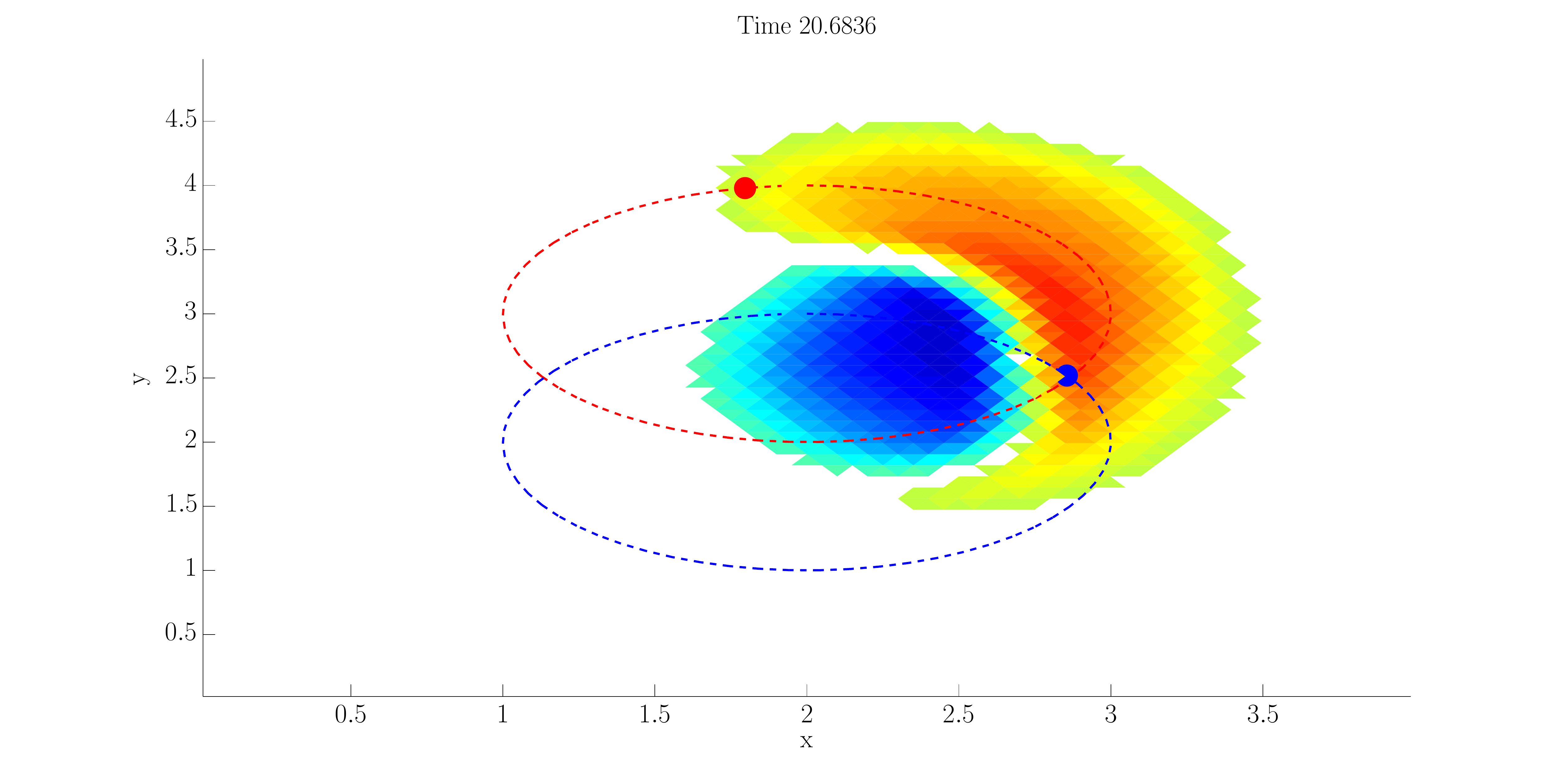}\\
  \includegraphics[width=0.45\textwidth, viewport=150 50 1350
  800]{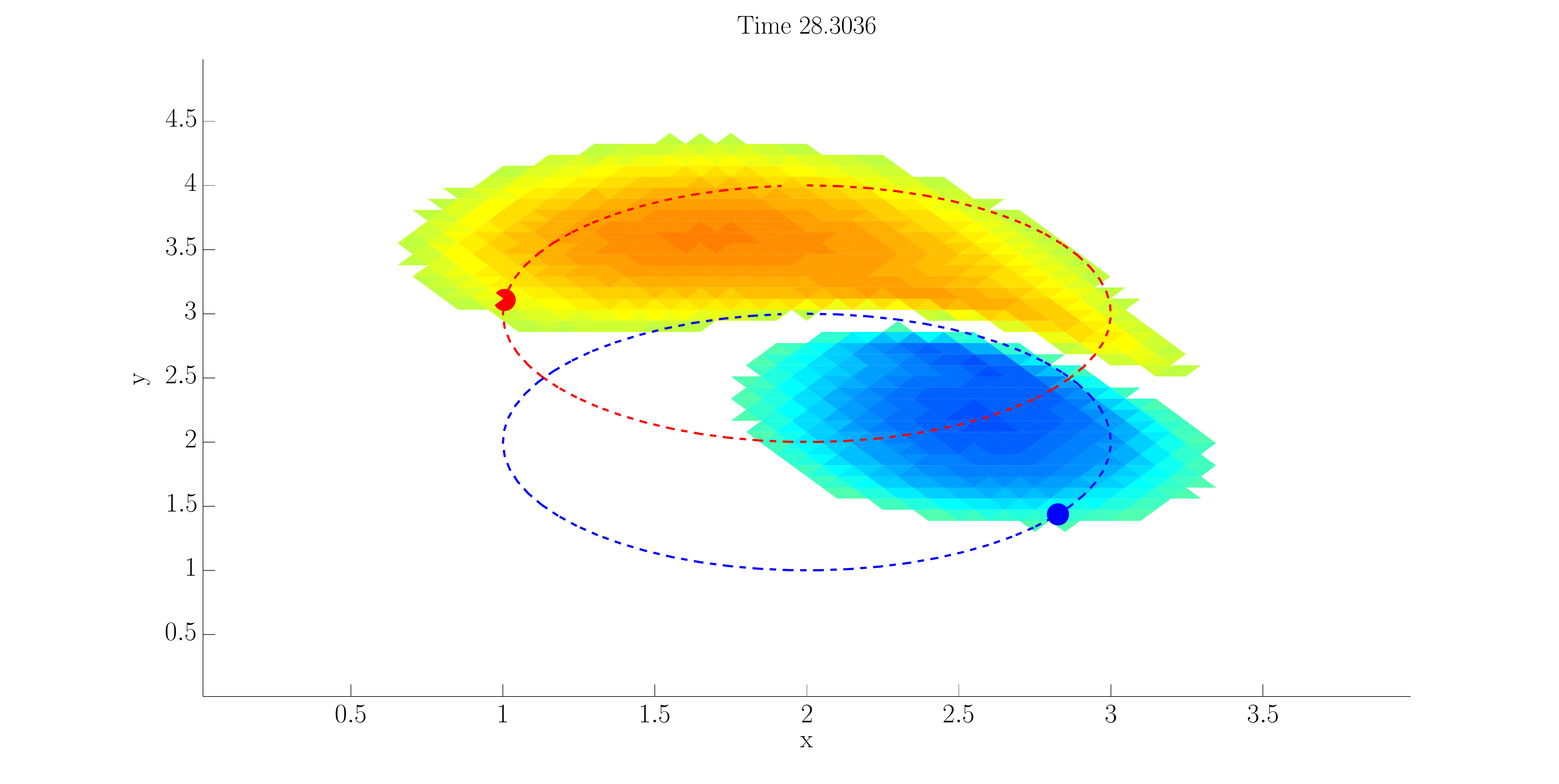}%
  \includegraphics[width=0.45\textwidth, viewport=150 50 1350
  800]{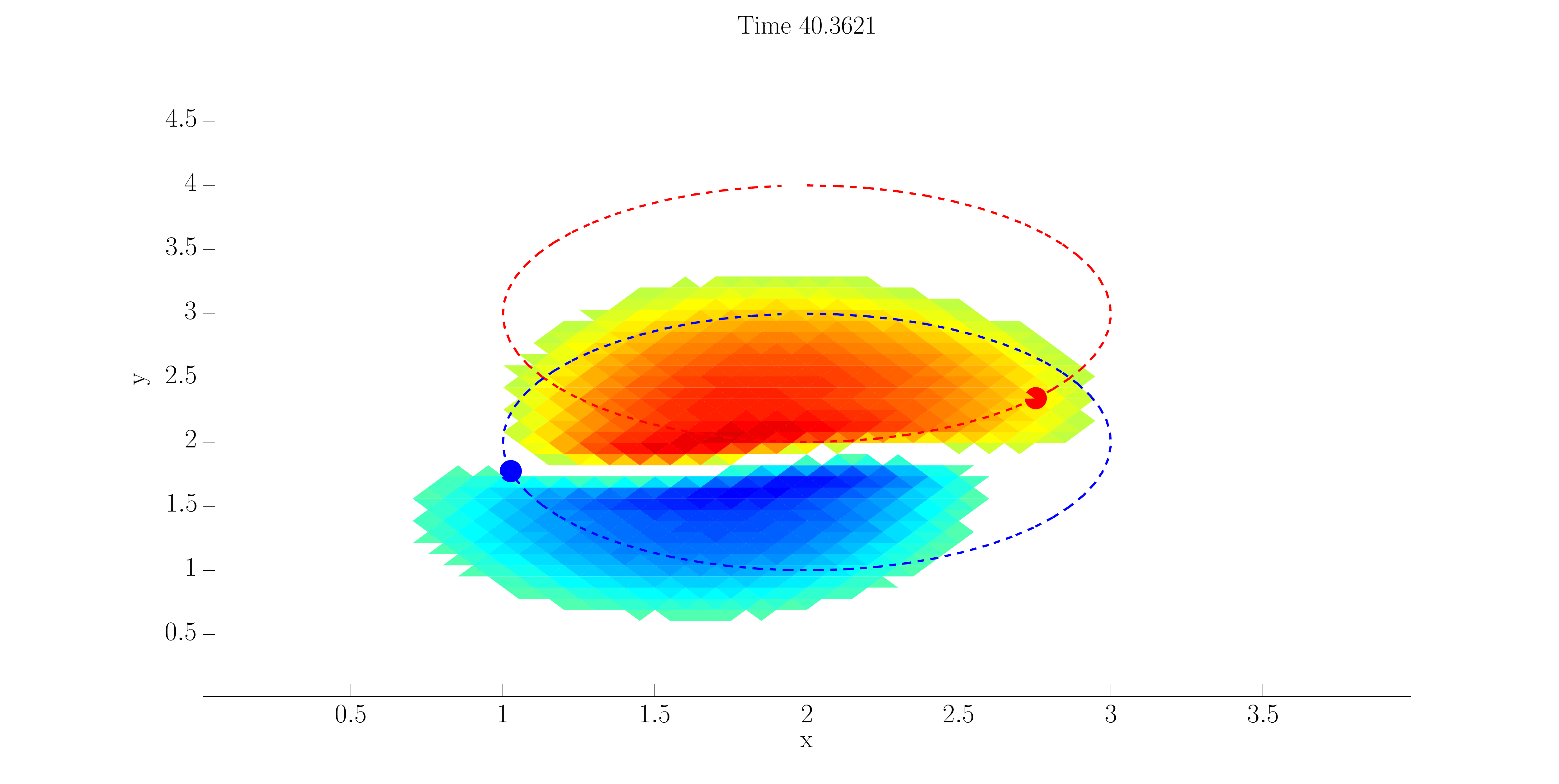}\\
  \caption{Plots of $\max\{\rho_1, \rho_2\}$ on the $(x,y)$ plane,
    where $(\rho_1, \rho_2)$ solve~\eqref{eq:density-tourists}--%
    \eqref{eq:p-tourists}--\eqref{eq:group-func-def}. The circles are
    the fixed trajectories of the guides. The blue refers to $i=1$,
    while the yellow/red to $i=2$. The blue guide moves clockwise, the
    other one counterclockwise. The choice~\eqref{eq:group-func-def}
    prevents the mixing of the two groups.}
  \label{fig:po_int_tourist}
\end{figure}

According to~\eqref{eq:p-tourists}--\eqref{eq:group-func-def}, the
groups walk towards their guides and come into contact at $t \approx
7.4$.  At $t \approx 20.7$, the blue guide is surrounded by the reds
and waits for his group.  Meanwhile, the red group bypasses the blues
and avoid the congestion.  At about $t \approx 28.3$, the groups are
separated, while they meet again at $t \approx 40.4$.

\subsection{Interacting Crowds and Vehicles}
\label{sub:car}

We consider two groups of pedestrians crossing a street at a
crosswalk, following~\cite{Raul1, Raul2}. The people near the
crosswalk reduce their speed and possibly stop if cars are near to the
crosswalk. At the same time, cars slow down and possibly stop as soon
as in front of them pedestrians are present.  The density $\rho^i(t,x)$,
for $i=1 \ldots, n$, describes
the $i$-th group of pedestrians. Each driver's position can thus be
identified through its scalar coordinate $p^k$, for $k = 1 \ldots, N$,
along the road. Without loss of generality, we assume that the road is
parallel to the vector %$\left[\begin{array}{@{}c@{}} 1\\0 \end{array}\right]$,
$[1, 0]^T$, with width $2h_R$, i.e. $\modulo{x_2 - \bar
  x_2} \leq h_R$. Therefore, we have
\begin{displaymath}
  d = 2\,,\quad
  n = 2\,,\quad
  m = N\,,\quad
  \ell = N \,.
\end{displaymath}

The dynamics of the pedestrians is similar to that introduced
in~\cite{m3as-cgm, ColomboMagali}, namely
\begin{equation}
  \label{eq:Ex2_1}
  \partial_t \rho^i
  +
  \div
  \left[
    2 \rho^i \;
    (1-\rho^i) \;
    w^i (x,p) \, \left(V^i (x) - \mathcal{A}^i (\rho)\right)
  \right]
  =
  0,
  \qquad i=1,2\,.
\end{equation}
Here $w^i \in \C2(\reali^d \times \reali^m; \reali^+)$ describes the
interaction between the member of the $i$-th group located at $x$ and
the cars. $\mathcal{A}^i$ is chosen as in~\eqref{eq:A_tourist} and
models the interactions of pedestrians.  Finally, the vector field
${V}^i \in \C2(\reali^d; \reali^d)$ stands for the preferred
trajectories of the people.

The dynamics of cars along the road is described by the Follow The
Leader model
\begin{equation}
  \label{eq:ftl}
  \left\{
    \begin{array}{rcl}
      \dot{p}^k
      & = &
      g\!\left(\left(\mathcal{B} (\rho)\right) (p^k)\right) \;
      u(p^{k+1} - p^{k}),
      \qquad k=1, \ldots, m-1,
      \\
      \dot p^m
      & = &
      v_L (t),
    \end{array}
  \right.
\end{equation}
where the non increasing function $g \in \C2 (\reali; [0,1])$
describes the slowing of cars when near to pedestrians while the non
decreasing function $u \in \C2 (\reali; [0,1])$ vanishes on $\reali^-$
and describes the usual drivers' behavior in Follow The Leader
models. The assigned function $v_L = v_L (t)$ is the speed of the
\emph{leader}, i.e., of the first vehicle. For simplicity, we assume
that the initial position of the first car is after the crosswalk so
that its subsequent dynamics is independent from the crowds.

The present model fits into the framework presented in
Section~\ref{sec:AR} by setting
\begin{equation}
  \label{eq:cars-func-def-2}
  \begin{array}{@{}r@{\,}c@{\,}lr@{\,}c@{\,}l@{}}
    q^1(\rho^1) & = & 2\rho^1(1-\rho^1),
    &
    q^2(\rho^2) & = & 2\rho^2(1-\rho^2),
    \\
    v^1(t,x,A,p) & = & w^1(x,p) \left(V^1(x) - A\right),
    &
    v^2(t,x,A,p) & = & w^2(x,p) \left(V^2(x) - A\right),
    \\
    \mathcal{A}^1(\rho)
    & = &
    \displaystyle
    \sum_{j=1}^2
    \frac{\epsilon_{1j} \, \nabla_x (\rho^j\ast \eta)}{\sqrt{1+\norma{\nabla_x
          (\rho^j\ast\eta)}_{\reali^2}^2}},
    &
    \mathcal{A}^2(\rho)
    & = &
    \displaystyle
    \sum_{j=1}^2
    \frac{\epsilon_{2j} \, \nabla_x (\rho^j\ast \eta)}{\sqrt{1+\norma{\nabla_x
          (\rho^j\ast\eta)}_{\reali^2}^2}},
  \end{array}
\end{equation}
\begin{equation}
  \label{eq:cars-func-def-2bis}
  \begin{array}{@{}rcl@{}}
    F_k (t,p,B)
    & = &
    \left\{
      \begin{array}{l@{\qquad\qquad}r@{\;}c@{\;}l@{}}
        g (B) \; u (p_1^{k+1} - p_1^k),
        & k & = & 1, \ldots, m-1,
        \\
        v_L (t),
        & k & = & m,
      \end{array}
    \right.
    \\
    \left(\mathcal{B} (\rho)\right) (p^k)
    & = &
    \displaystyle
    \int_{\reali^2}
    \left(\rho^1 (x) + \rho^2 (x)\right) \;
    \bar\eta \! \left(x - \left[
        \begin{array}{@{}c@{}}
          p^k\\ \bar x_2
        \end{array}
      \right]\right)
    \, \d{x},
    \quad
    k=1, \ldots, m\,.\!\!\!\!
  \end{array}
\end{equation}

\begin{proposition}
  \label{prop:Ex2}
  Assume $\eta, \bar\eta \in \Cc2(\reali^2; \reali^+)$, $w^i \in
  \C2(\reali^2 \times \reali^{2N}; \reali^+)$, $V^i \in
  \C2\left(\reali^2; \reali^2\right)$, $v_L \in \L1 (\reali^+;
  \reali^+)$ and $g,u \in \C2\left(\reali;[0,1]\right)$. Then, the
  functions defined
  in~\eqref{eq:cars-func-def-2}--\eqref{eq:cars-func-def-2bis} satisfy
  \textbf{(v.1)}, \textbf{(F)}, ($\boldsymbol{\mathcal{A}}$),
  ($\boldsymbol{\mathcal{B}}$) and \textbf{(q)}. In particular,
  Corollary~\ref{cor:zero} applies to~\eqref{eq:Ex2_1}-\eqref{eq:ftl}.
\end{proposition}
\noindent The proof is deferred to Section~\ref{sec:technical}.

As a specific example we consider the spatial domain $\mathcal{D} =
[0,1] \times [0,1]$, with a road occupying the region $\mathcal{R} =
[0,1] \times [0.45, 0.55]$ (so that $\bar x_2 = 0.5$ and $h_R = 0.05$)
and the crosswalk $\mathcal{C} = [0.4, 0.6] \times [0.45,
0.55]$. Therefore, pedestrians may walk in $\mathcal{C} \cup
(\mathcal{D} \setminus \mathcal{R})$, while cars travel along
$\mathcal{R}$ from left to right. The $\rho^1$ population targets the
bottom boundary $[0,1] \times \{0\}$, while $\rho^2$ points towards
the top boundary $[0,1] \times \{1\}$, see
Figure~\ref{fig:po_int_car}. No individual is allowed to cross the
road aside the crosswalk.

The vector $V^1 (x)$, respectively $V^2 (x)$, is chosen with norm $1$
and tangent to the geodesic path at $x$ for the population $1$,
respectively $2$. In general, these vectors can be computed as
solutions to the eikonal equation and their regularity depends on the
geometry of the domain\cite{sethian1999fast}.

%%%%%%%%%%%%%%%%%%%%%%%%

First, for $\alpha_1 < \alpha_2$, we introduce the smooth function
$\beta_{\alpha_1, \alpha_2} \in \C\infty (\reali; [0,1])$ defined as\\
\begin{minipage}{0.35\linewidth}
  \includegraphics[width=\textwidth]{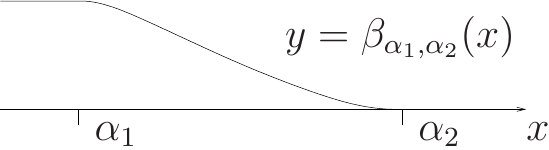}
\end{minipage}%
\begin{minipage}{0.65\linewidth}
  \begin{displaymath}
    \hfill \beta_{\alpha_1, \alpha_2}(z) =
    \left\{
      \begin{array}{@{}lr@{\,}c@{\,}l@{}}
        1, & z & < & \alpha_1,
        \\
        \exp \left[
          1 - \left(1 - \left(
              \frac{z-\alpha_1}{\alpha_2-\alpha_1} \right)^2 \right)^{-1}
        \right],
        & z & \in &[\alpha_1, \alpha_2],
        \\
        0, & z & > & \alpha_2\,.
      \end{array}
    \right.
  \end{displaymath}
\end{minipage}\\
For $i=1,2$ we choose
\begin{eqnarray*}
  w^i(x,p)
  & = &
  1
  -
  \left(1 - \beta_{h_R, h_R + \epsilon_\gamma}(\modulo{x_2 - \bar x_2})\right)
  \!
  \left[
    1 - \prod_{l=1}^3  \beta_{r_i,r_a} \left(\hat{\eta}^i
      \left(x, \left[
          \begin{array}{@{}c@{}}
            p^l\\ \bar x_2
          \end{array}
        \right] \right)\right)
  \right],
  \quad
  \begin{array}{@{}r@{\,}c@{\,}l@{}}
    \bar x_2 & = & 0.5,
    \\
    h_R & = & 0.05,
    \\
    \epsilon_\gamma & = &0.001,
    \\
    r_i & = & 0.1,
    \\
    r_a & = & 0.8,
    \\
  \end{array}
  \\
  \hat{\eta}^i (x,p)
  & = &
  \left\{
    \begin{array}{@{}ll@{}}
      \eta_3 \!\left( (p-x) \cdot V^i(x), r_v\right) \;
      \eta_3 \! \left((p-x) \cdot V^i(x)^\perp, r_v\right),
      & (p-x) \cdot V^i(x)>0,
      \\
      \eta_3 \! \left((p-x) \cdot V^i(x), r_{vb}\right) \;
      \eta_3 \! \left((p-x) \cdot V^i(x)^\perp, r_v\right), \quad
      & \mbox{otherwise,}
    \end{array}
  \right.\\
%  \text{with}\\
  \eta_3(z,r) & = &
  \left\{
    \begin{array}{@{}ll@{}}
      \exp \left(- \frac{z^2}{r^2 - z^2}\right),
      & z \in [-r, r],
      \\
      0,
      & \mbox{otherwise,}
    \end{array}
  \right.
  \qquad\qquad\qquad\qquad\qquad
  \begin{array}{@{}r@{\,}c@{\,}l@{}}
    r_v & = & 0.15,
    \\
    r_{vb} & = & 0.0015.
  \end{array}
\end{eqnarray*}
Here, $\hat\eta_i$ describes the region considered by each pedestrian
in reacting to cars. For instance, cars behind a pedestrian are
ignored when at a distance greater than $r_{vb}$, while cars in front
of the pedestrian are considered up to a distance $r_v$. Outside the
interval defined by the threshold parameters $r_i$ and $r_a$, the
pedestrians' sensitivity to cars is amplified.

The convolution kernel in the nonlocal operators $\mathcal{A}^1$ and
$\mathcal{A}^2$ are
\begin{displaymath}
  \tilde \eta_r (x)
  =
  \left\{
    \begin{array}{@{}ll@{}}
      \exp
      \left(
        - \frac{5 {x_1}^2}{r^2 - {x_1}^2} - \frac{5 {x_2}^2}{r^2 - {x_2}^2}
      \right),
      &
      x \in [-r, r]^2,
      \\
      0,
      &
      \mbox{otherwise,}
    \end{array}
  \right.
  \qquad
  \eta (x)
  =
  \frac{\tilde \eta_r (x)}{\int_{\reali^2} \tilde \eta_r (x)\d{x}},
  \quad\qquad
  r = 0.05 \,.
\end{displaymath}
with interaction parameters
\begin{displaymath}
  \epsilon_{11} = 0.1\,, \qquad \epsilon_{22} = 0.1\,, \qquad
  \epsilon_{12} = 0.7\,, \qquad \epsilon_{21} = 0.7 \,.
\end{displaymath}

In the Follow The Leader model, we choose $N=3$ vehicles and let
\begin{displaymath}
  v_L (t)=1 \,, \qquad
  g (B) = \beta_{r_j,r_b} (B),
  \quad \mbox{ and } \quad
  u (\xi) = 1 - \beta_{H, 10H} (\xi)^{K},
  \quad
  \mbox{ with } \quad
  \begin{array}{@{}r@{\,}c@{\,}l@{}}
    r_j & = & 0.125,
    \\
    r_b & = & 0.5,
    \\
    H & = & 0.167,
    \\
    K & = & 50.
  \end{array}
\end{displaymath}
The microscopic model for vehicles is completed by the convolution
kernel in the nonlocal operator $\mathcal{B}$
\begin{displaymath}
  \tilde{\eta_2} (x) =
  \left\{
    \begin{array}{ll}
      \eta_R(x), & x_1 > 0,
      \\
      \eta_{R'}(x_1,0) \; \eta_R(0,x_2),  & \mbox{otherwise,}
    \end{array}
  \right.
  \qquad\quad
  \bar\eta (x)
  =
  \frac{\tilde \eta_2 (x)}{\int_{\reali^2} \tilde \eta_2 (x)\d{x}},
  \qquad\quad
  \begin{array}{@{}r@{\,}c@{\,}l@{}}
    R & = & 0.045,
    \\
    R' & = & 0.0045 \,.
  \end{array}
\end{displaymath}

As initial condition we prescribe
\begin{displaymath}
  \begin{array}{rcl}
    \rho_0^1 & = & \chi_{\strut [0.1,0.9]\times[0.7,0.9]},
    \\
    \rho_0^2 & = & 0.5 \, \chi_{\strut [0.1,0.9]\times[0.1,0.3]},
  \end{array}
  \qquad \mbox{ and } \qquad
  \begin{array}{rcl}
    p_o^1 & = & 0.000,
    \\
    p_o^2 & = & 0.333,
    \\
    p_o^3 & = & 0.667\,.
  \end{array}
\end{displaymath}

In Figure~\ref{fig:po_int_car} the maximal density $\rho_m =
\max(\rho^1,\rho^2)$ of the two groups is shown. The first group is
illustrated in blue and the second one in red.
\begin{figure}[ht!]
  \centering
  \includegraphics*[width=0.5\textwidth, viewport=125 30 1350
  800]{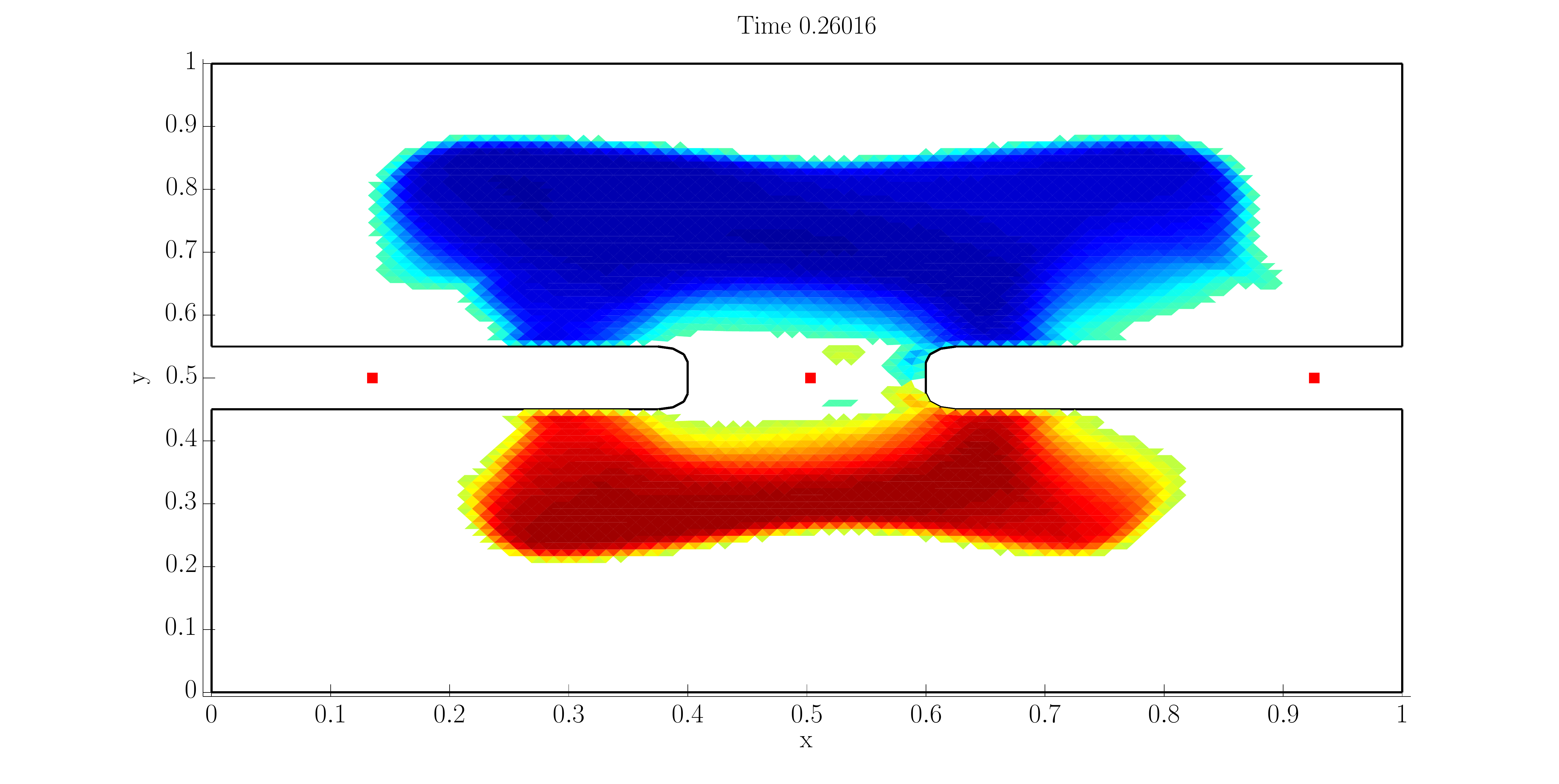}%
  \includegraphics*[width=0.5\textwidth, viewport=125 30 1350
  800]{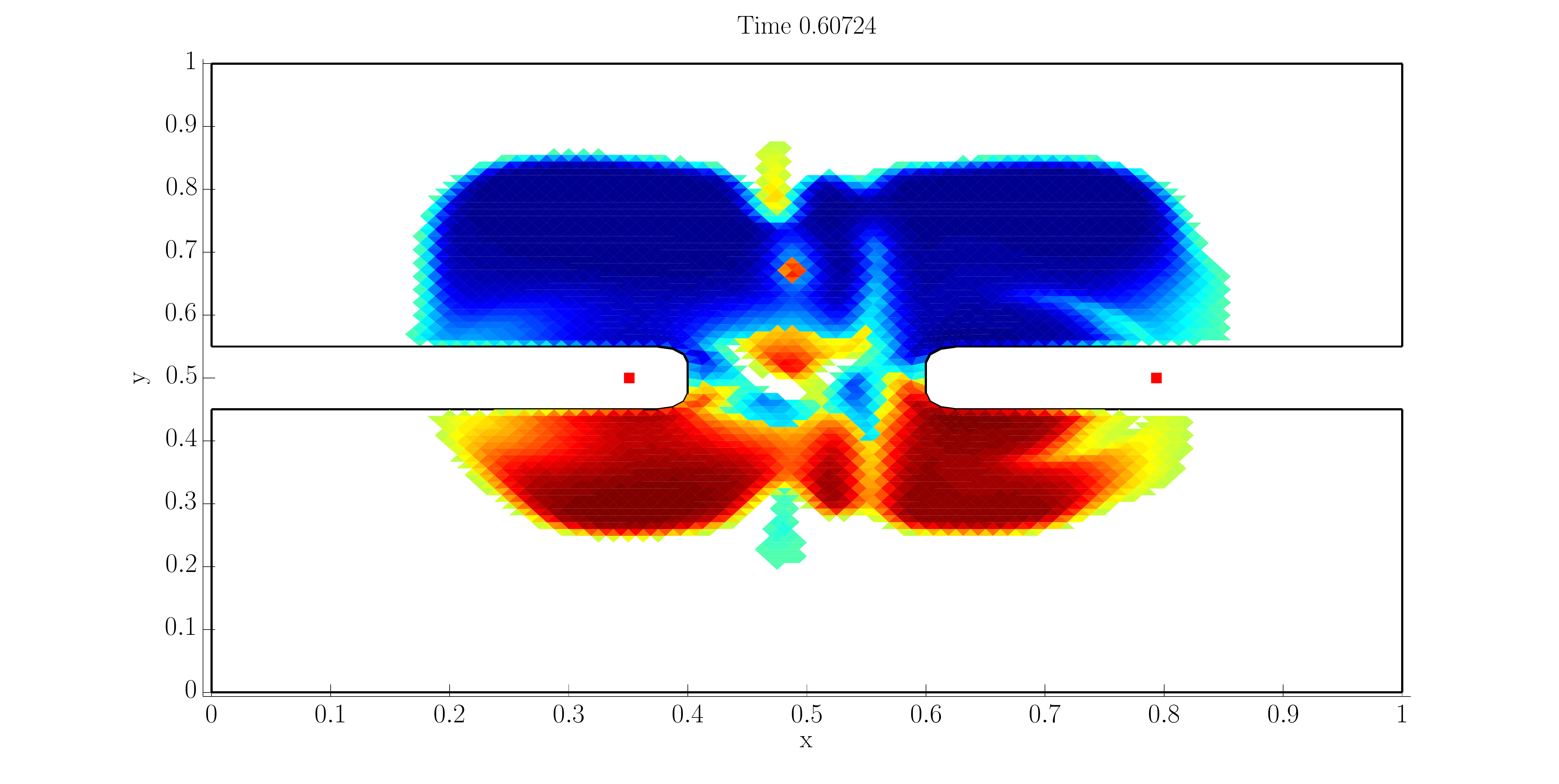}\\
  \includegraphics*[width=0.5\textwidth, viewport=125 30 1350
  800]{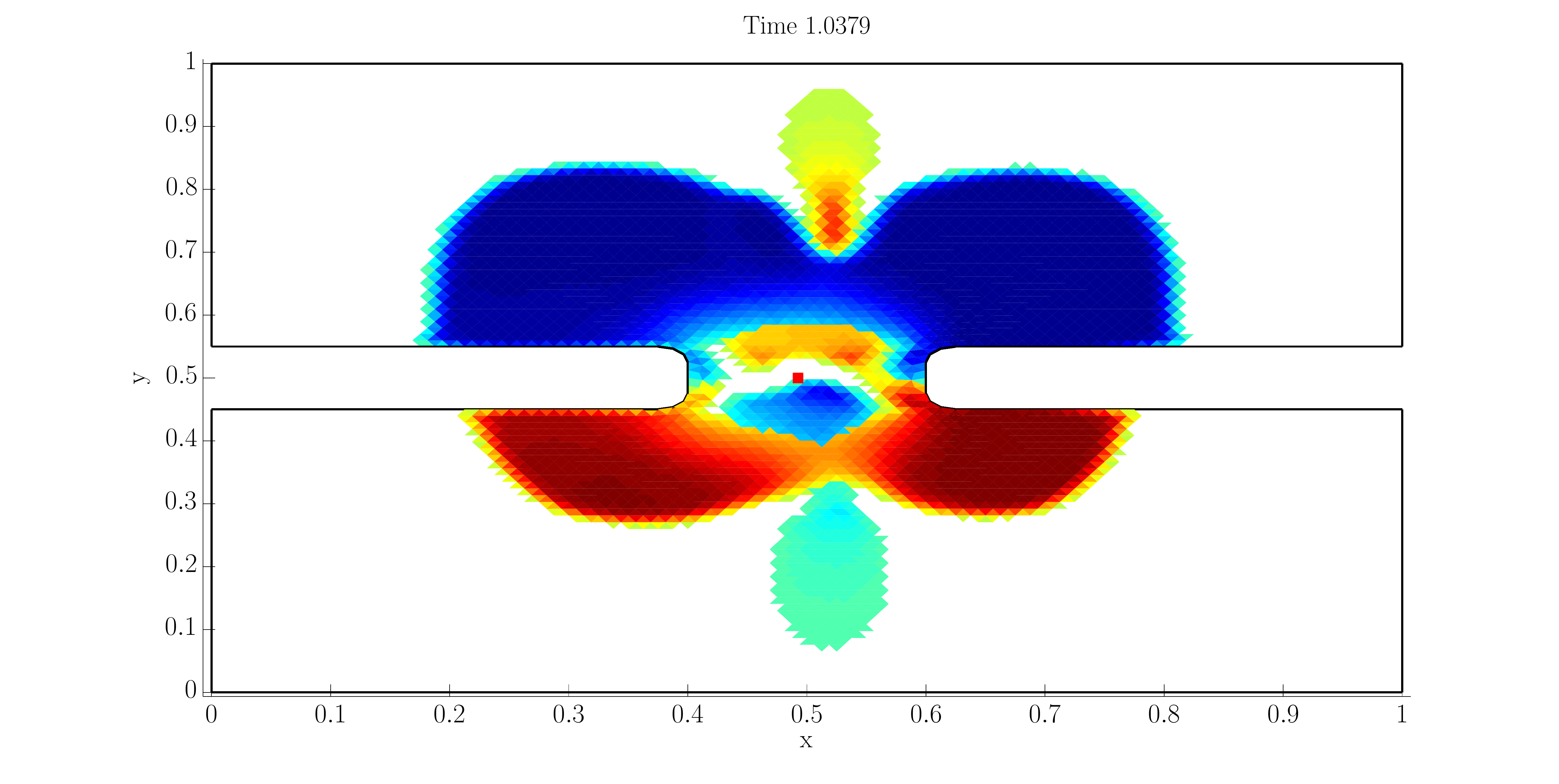}%
  \includegraphics*[width=0.5\textwidth, viewport=125 30 1350
  800]{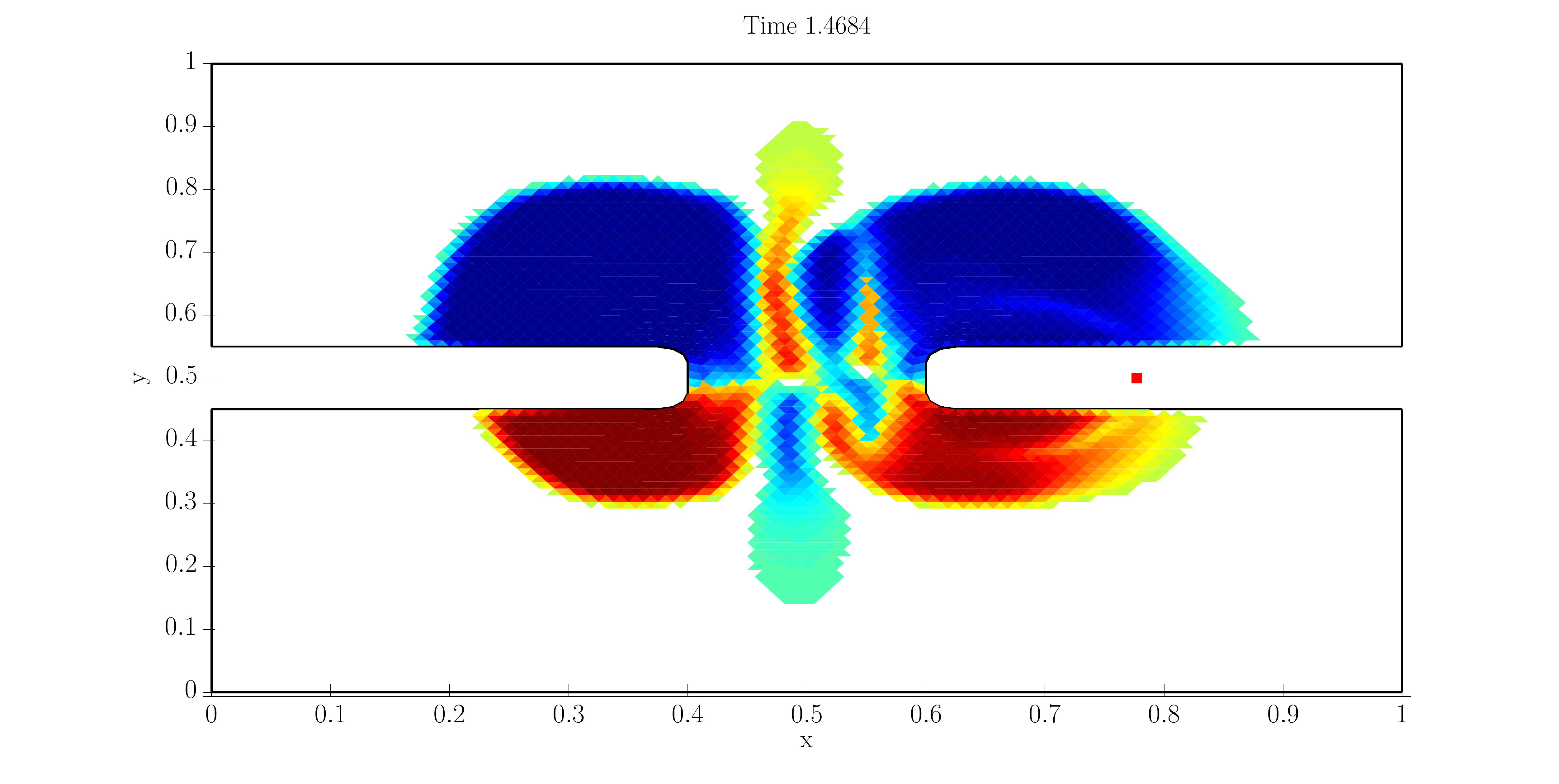}\\
  \caption{Plots of $\max\{\rho_1, \rho_2\}$ on the $(x,y)$
    plane, where $(\rho_1, \rho_2)$ solve~\eqref{eq:Ex2_1}--%
    \eqref{eq:ftl}--\eqref{eq:cars-func-def-2}--%
    \eqref{eq:cars-func-def-2bis}. The blue population $\rho_1$ moves
    downward; the red one, $\rho_2$. upward. Cars are represented by
    the red dots along the road. Above: left, pedestrians wait until
    the second car has passed the crosswalk; right, pedestrians cross
    the road and form lanes. Bottom: left, pedestrians wait until the
    third car has passed the crosswalk; right, pedestrians cross the
    road and form lanes.}
  \label{fig:po_int_car}
\end{figure}
\begin{figure}[ht!]
  \centering
  \includegraphics*[width=0.5\textwidth, viewport=125 30 1350
  800]{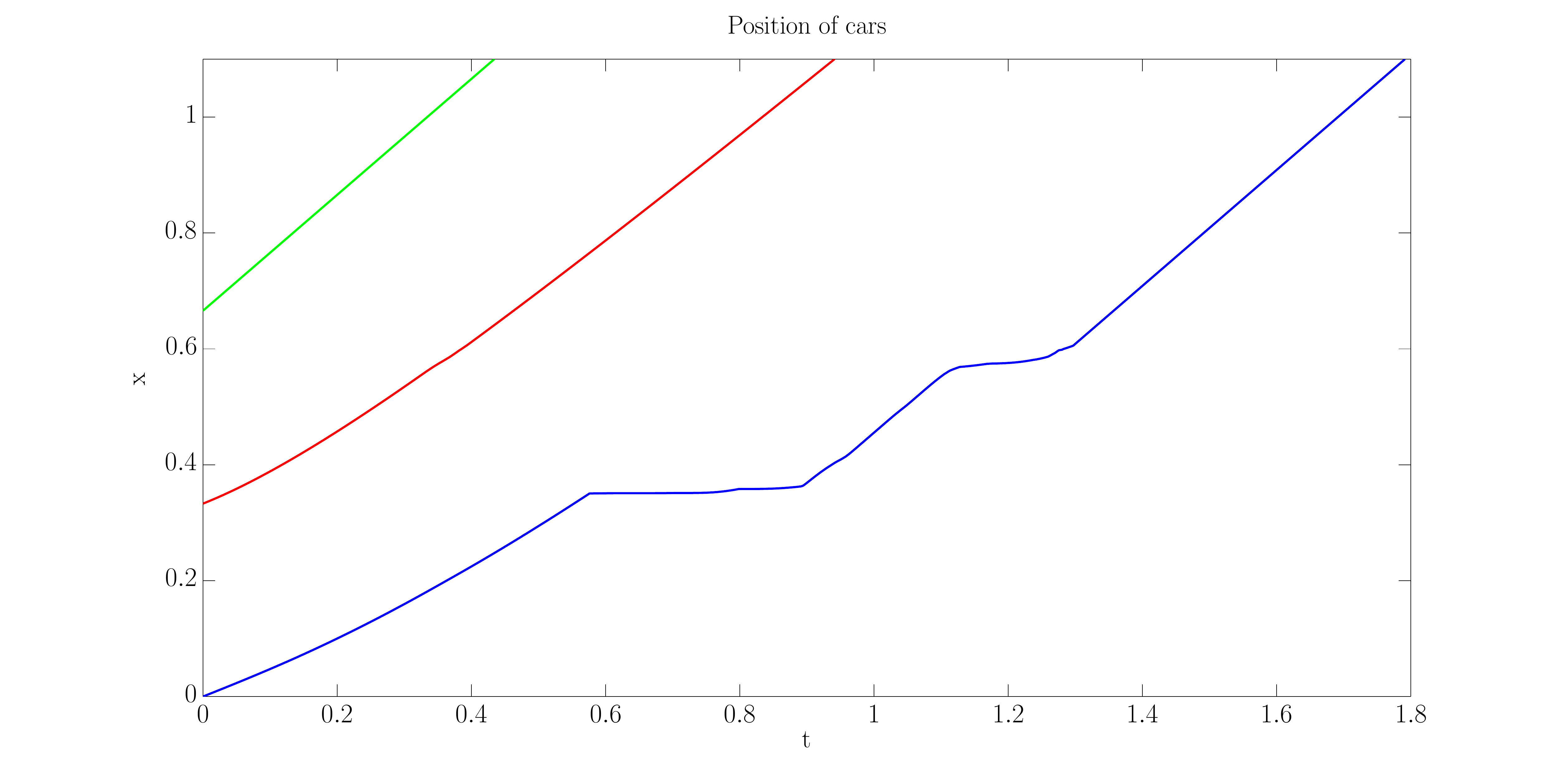}%
  \includegraphics*[width=0.5\textwidth, viewport=125 30 1350
  800]{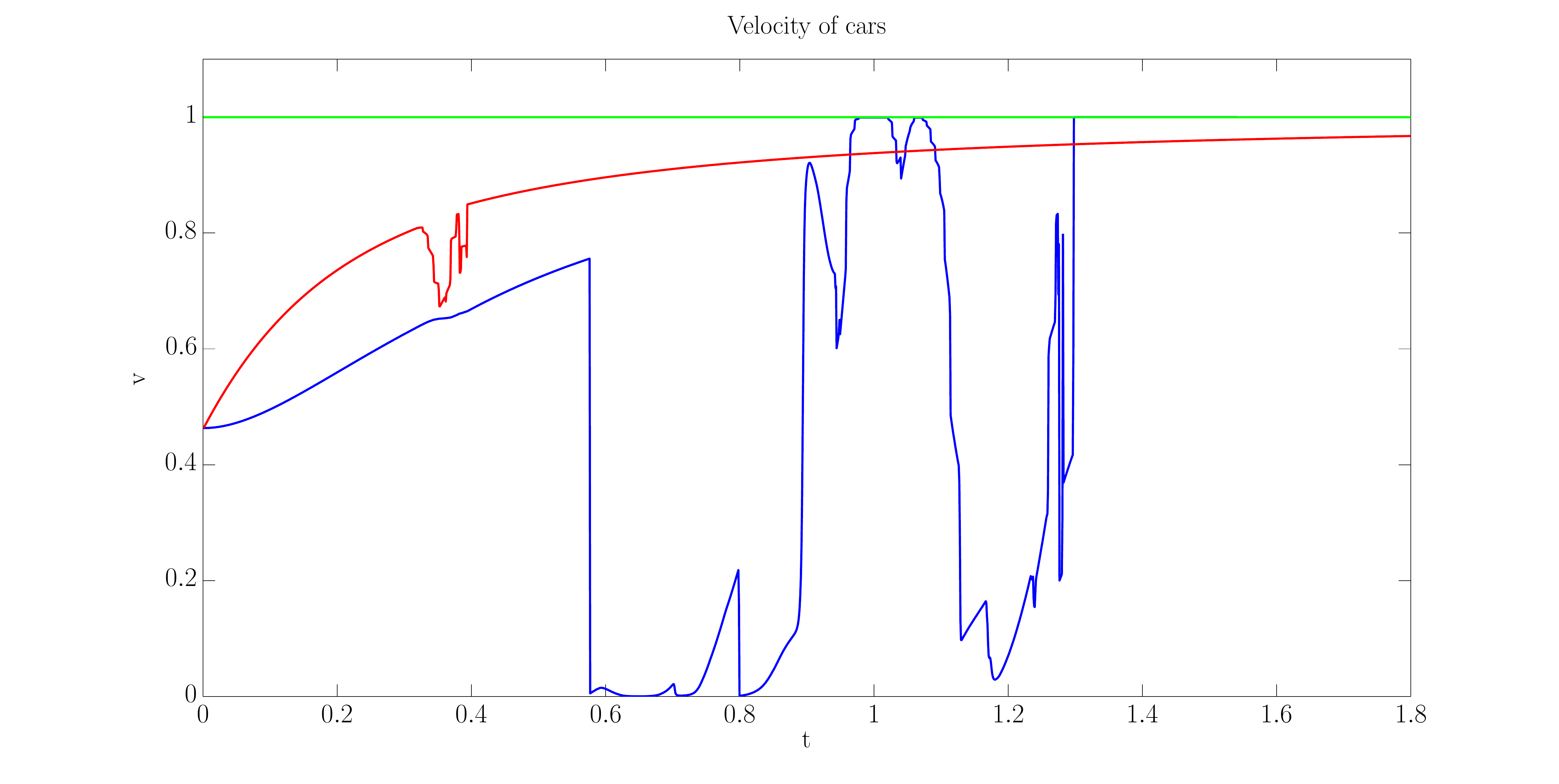}\\
  \caption{Positions, left, and velocities, right, of the cars in the
    solution to~\eqref{eq:Ex2_1}--%
    \eqref{eq:ftl}--\eqref{eq:cars-func-def-2}--%
    \eqref{eq:cars-func-def-2bis} as a function of time, on the
    horizontal axes. The green lines refer to $p_3$, the red ones to
    $p_2$ and the blue one to $p_1$.}
  \label{fig:po_int_car_v}
\end{figure}
Initially the pedestrians start walking towards the crosswalk and the
cars can drive freely. % because there are no pedestrians on the road.
The first car has maximal speed $1$ and the other ones adapt their
speed according to the distance to their leading car, see
Figure~\ref{fig:po_int_car_v}. At time $t\approx 0.2$ the second car
is in the middle of the crosswalk and only few pedestrians try to
cross the road (Figure~\ref{fig:po_int_car}, top left). When the car
has left the crosswalk, the pedestrians start walking and form lanes
in order to pass through the other group (Figure~\ref{fig:po_int_car},
top right).  When the third car approaches the crosswalk, the
pedestrians in front of the crosswalk stop while those on the road can
continue their way (Figure~\ref{fig:po_int_car}, bottom left). As the
street is not cleared immediately, the car almost has to stop (see
Figure~\ref{fig:po_int_car_v}, left). When it has passed the
pedestrians can walk again until all have reached their exits
(Figure~\ref{fig:po_int_car}, bottom right).

\subsection{The Police Separates Conflicting Hooligans}
\label{sub:hools}

In this example we consider $n=2$ groups of conflicting hooligans and
their interaction with police officers in a $d=2$ dimensional
region. For the hooligans we use a model of the form
\begin{equation}
  \label{eq:Ex3_1}
  \partial_t \rho^i
  +
  \div
  \left[
    \rho^i (1-\rho^i)
    \left(- w^i(x,p) + \mathcal{A}^i (\rho)  \right)
  \right]
  =
  0,
  \qquad i=1,2\,,
\end{equation}
where $\rho^i$ is the density of the $i$-th group.  Here $w^i \in
\C2(\reali^d \times \reali^m; \reali^d)$ describes the preferred
direction of the hooligans belonging to the $i$-th group and located at
$x$ in presence of the police officers $p^1, \ldots, p^N$.  The terms
$\mathcal{A}^1 (\rho), \mathcal{A}^2 (\rho)$ modify the hooligans'
direction according to their distribution in space.  The movement of
the $N$ police officers is described by the ODEs
\begin{equation}
  \label{eq:Ex3_2}
  \dot{p}^k = I_k(p) + \mathcal B_k(\rho),
  \qquad k=1, \ldots, N,
\end{equation}
where $p^k = [p^k_1, p^k_2]^T$ denotes the position in $\reali^2$ of
the $k$-th policeman; so we set $m = 2N$.  The term $I_k \in \C0
(\reali^{2N};\reali^2)$ avoids concentrations of officers at the same
place, while the term $\mathcal B_k(\rho)$ takes into consideration
the distribution of the hooligans.

The present model fits in the framework presented in
Section~\ref{sec:AR} by setting
\begin{equation}
  \label{eq:cars-func-def}
  \begin{array}{rcl@{\qquad}rcl}
    q^1(\rho) & = & \rho (1-\rho),
    &
    q^2(\rho) & = & \rho (1-\rho),
    \\
    v^1(t,x,A,p) & = &  - w^1(x,p) + A,
    &
    v^2(t,x,A,p) & = & - w^2(x,p) + A,
    \\
    F_k (t,p,B) & = &  I_k(p) + B_k,
  \end{array}
\end{equation}
\begin{equation}
  \label{eq:hooligans-2}
  \begin{array}{@{}rcl@{}}
    \displaystyle
    \mathcal{A}^1(\rho)
    & = & \displaystyle
    \frac{\epsilon_{11} \; \eta\ast(\rho^1 - \bar \rho) \;
      \nabla_x (\rho^1\ast \eta)}{\sqrt{1+\Vert\eta\ast(\rho^1 - \bar \rho)
        \nabla_x (\rho^1\ast \eta)\Vert^2}} +
    \frac{\epsilon_{12} \; \eta\ast(\rho^2-\rho^1) \nabla_x (\rho^2\ast
      \eta)}{\sqrt{1+\Vert\eta\ast(\rho^2-\rho^1) \nabla_x (\rho^2\ast
        \eta)\Vert^2}},
    \vspace{.2cm}
    \\
    \displaystyle
    \mathcal{A}^2(\rho)
    & = & \displaystyle
    \frac{\epsilon_{22} \; \eta\ast(\rho^2 - \bar \rho) \;
      \nabla_x (\rho^2\ast \eta)}{\sqrt{1+\Vert\eta\ast(\rho^1 - \bar \rho)
        \nabla_x (\rho^2\ast \eta)\Vert^2}} +
    \frac{\epsilon_{21} \; \eta\ast(\rho^1-\rho^2) \nabla_x (\rho^1\ast
      \eta)}{\sqrt{1+\Vert\eta\ast(\rho^1-\rho^2) \nabla_x (\rho^1\ast
        \eta)\Vert^2}},
    \vspace{.2cm}
    \\ \displaystyle
    \mathcal B_k(\rho)(p)
    & = & \displaystyle
    \bar{\epsilon}_1
    \frac{1}{N}\sum_{j=1}^{N}
    \sum_{l\neq j}\frac{\nabla_x((\bar{\eta}\ast\rho^l)(\bar{\eta}\ast\rho^j))(p^k)}
    {\sqrt{1+\Vert\nabla_x((\bar{\eta}\ast\rho^l)(\bar{\eta}\ast\rho^j))(p^k)
        \Vert^2}}.
  \end{array}
\end{equation}
In~(\ref{eq:hooligans-2}), the operator $\mathcal A^i$ is composed by
two terms describing the attraction, respectively repulsion, between
members of the same, respectively different, group. Here, we introduce
a \emph{preferred density} $\bar \rho \in [0,1]$.  If the density of
one group is lower than $\bar \rho$, then members of that group tend
to move towards each other.  On the contrary, if the density is bigger
than $\bar \rho$, then they tend to disperse.  Moreover, the operator
$\mathcal A^i$ also models the fact that one group of hooligans aims
at attacking the other group as soon as it feels to be stronger.  On
the contrary, hooligans of a faction try to avoid the adversaries in
case they are less represented.

\begin{proposition}
  \label{prop:Ex3}
  Let $N \in \naturali \setminus \{0\}$.  Assume $\eta, \bar\eta \in
  \Cc2(\reali^2; \reali^+)$, $w^i \in \C2(\reali^2 \times \reali^{2N};
  \reali^2)$, $I_k \in (\C0 \cap \L\infty)(\reali^{2N};
  \reali^2)$. Then, the functions defined
  in~(\ref{eq:cars-func-def})--(\ref{eq:hooligans-2}) satisfy
  \textbf{(v.1)}, \textbf{(F)}, \textbf{($\boldsymbol{\mathcal{A}}$)},
  \textbf{($\boldsymbol{\mathcal{B}}$)} and \textbf{(q)}. In
  particular, Corollary~\ref{cor:zero} applies
  to~(\ref{eq:cars-func-def})--(\ref{eq:hooligans-2}).
\end{proposition}
\noindent The proof is deferred to Section~\ref{sec:technical}.

As a specific example, in the computational domain $[0,1]^2$, we
consider $N=4$ policemen, so that $m = 8$, and the
parameters
\begin{displaymath}
  \epsilon_{11} = 0.5\,, \qquad \epsilon_{22} = 0.5 \,,\qquad
  \epsilon_{12} = 0.5\,, \qquad \epsilon_{21} = 0.5 \,,\qquad
  \bar \epsilon_1 = 0.4\,, \qquad \bar \rho = 0.5\,,
\end{displaymath}
with the functions
\begin{displaymath}
  \begin{array}{@{}rclrclrcl@{}}
    w^1(x,p)
    & = &
    \displaystyle
    \frac{\epsilon_3}{N}
    \sum_{j=1}^{N}
    \hat{\eta}(x-p^j)
    \left[
      \begin{array}{@{}c@{}}
        0 \\ -1
      \end{array}
    \right],
    &
    \epsilon_3 & = & 0.1,
    \\
    w^2(x,p)
    & = &
    \displaystyle
    \frac{\epsilon_4}{N}
    \sum_{j=1}^{N}
    \hat{\eta}(x-p^j)
    \left[
      \begin{array}{@{}c@{}}
        0 \\ 1
      \end{array}
    \right],
    &
    \epsilon_4 & = & 0.1,
    \\
    I_k(p)
    & = &
    \displaystyle
    \frac{\bar{\epsilon}_2}{N}
    \sum_{j=1}^{N}
    \frac{\nabla_x \tilde{\eta}(p^j-p^k)}{\sqrt{1+\Vert\nabla_x
        \tilde{\eta}(p^j-p^k) \Vert^2}},
    \quad
    &
    \bar \epsilon_2 & = & 0.2\,,
    &
    k & = & 1,\ldots,N \,.
  \end{array}
\end{displaymath}
Moreover, we let
\begin{displaymath}
  \begin{array}{rcl@{\mbox{ with }}r}
    \eta_{r} (x)
    & = &
    \left\{
      \begin{array}{ll@{}}
        \exp\left(
          -\frac{5{x_1}^2}{r^2 - {x_1}^2}
          -\frac{5{x_2}^2}{r^2 - {x_2}^2}
        \right),
        &
        x \in [-r, r]^2,
        \\
        0, & \mbox{otherwise,}
      \end{array}
    \right.
    &
    \begin{array}{rcl@{}}
      \eta (x)
      & = &
      \frac{\eta_{0.1} (x)}{\int_{\reali^2} \eta_{0.1} (x) \d{x}} \,,
      \\
      \hat\eta (x)
      & = &
      \frac{\eta_{0.15} (x)}{\int_{\reali^2} \eta_{0.15} (x) \d{x}} \,,
    \end{array}
    \\
    \tilde \eta_r (x)
    & = &
    \left\{
      \begin{array}{ll@{}}
        \left[
          \left(1 - \frac{x_1}{r})^2\right)
          \left(1 - \frac{x_2}{r})^2\right)
        \right]^3,
        & x \in [-r,r]^2,
        \\
        0,
        & \mbox{otherwise,}
      \end{array}
    \right.
    &
    \begin{array}{rcl@{}}
      \bar\eta (x)
      & = &
      \frac{\tilde\eta_{0.1} (x)}{\int_{\reali^2} \tilde\eta_{0.1} (x) \d{x}}\,, \;
      \\
      \tilde\eta (x)
      & = &
      \frac{\tilde\eta_{0.2} (x)}{\int_{\reali^2} \tilde\eta_{0.2} (x) \d{x}} \,. \;
    \end{array}
  \end{array}
\end{displaymath}

For the numerical example, the initial conditions are
\begin{equation}
  \label{eq:Ex3_ID}
  \begin{array}{rcl}
    \rho_0^1 & = &0.9 \; \chi_{\strut [0.25,0.75]\times[0.2,0.5]},
    \\
    \rho_0^2 & = & 0.7 \; \chi_{\strut [0.25,0.75]\times[0.5,0.8]},
  \end{array}
  \qquad \mbox{ and } \qquad
  \begin{array}{rcl}
    p_o^1 & = &[0.1,0.7]^T,
    \\
    p_o^2 & = &[0.9,0.3]^T,
    \\
    p_o^3 & = &[0.1,0.4]^T,
    \\
    p_o^4 & = &[0.9,0.7]^T \,.
  \end{array}
\end{equation}
In the pictures below the density of the two groups are plotted
separately.  The police officers are indicated by green circles.
\begin{figure}[ht!]
  \centering
  \includegraphics*[width=0.45\textwidth, viewport=125 30 1350
  800]{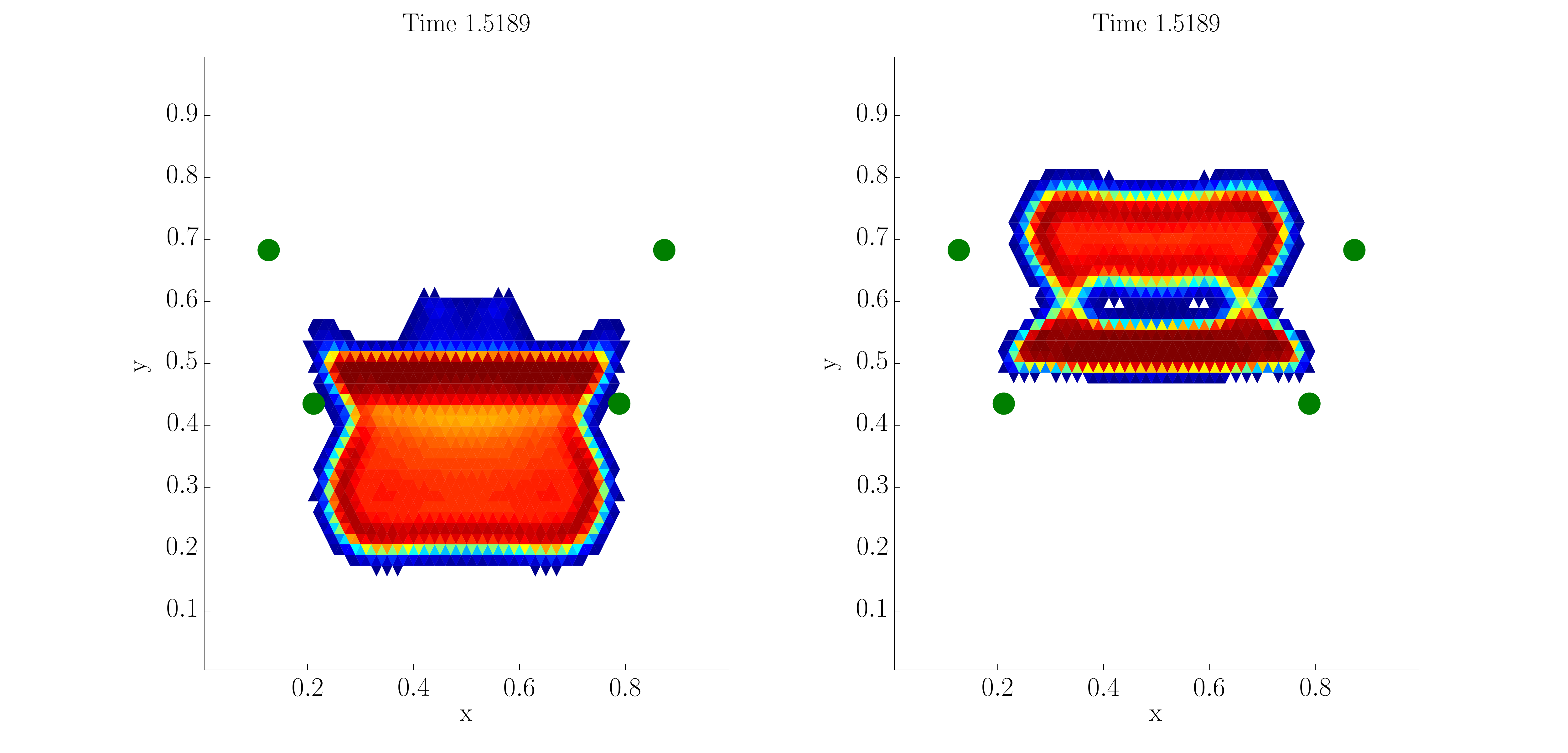}\hfill%
  \includegraphics*[width=0.45\textwidth, viewport=125 30 1350
  800]{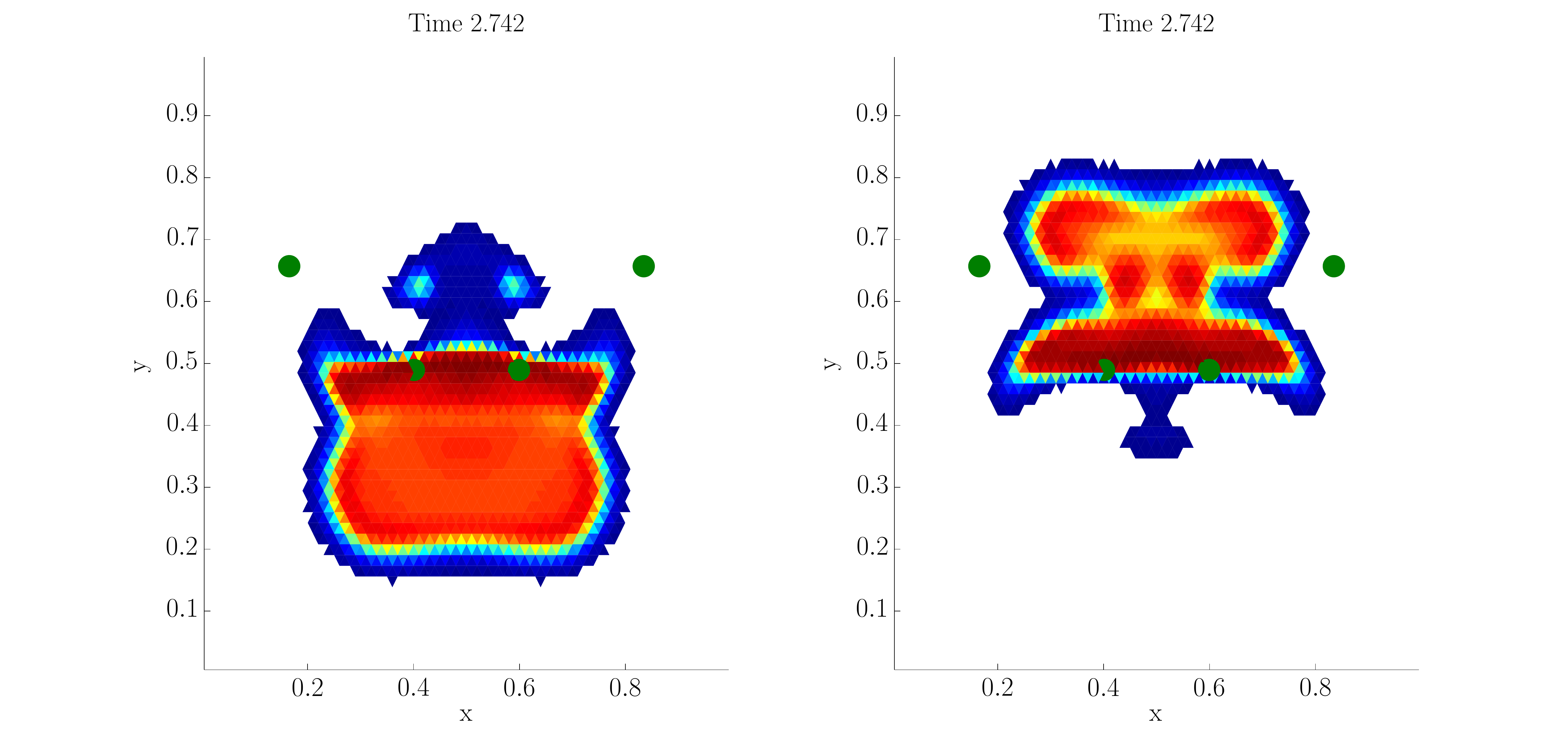}\\
  \includegraphics*[width=0.45\textwidth, viewport=125 30 1350
  800]{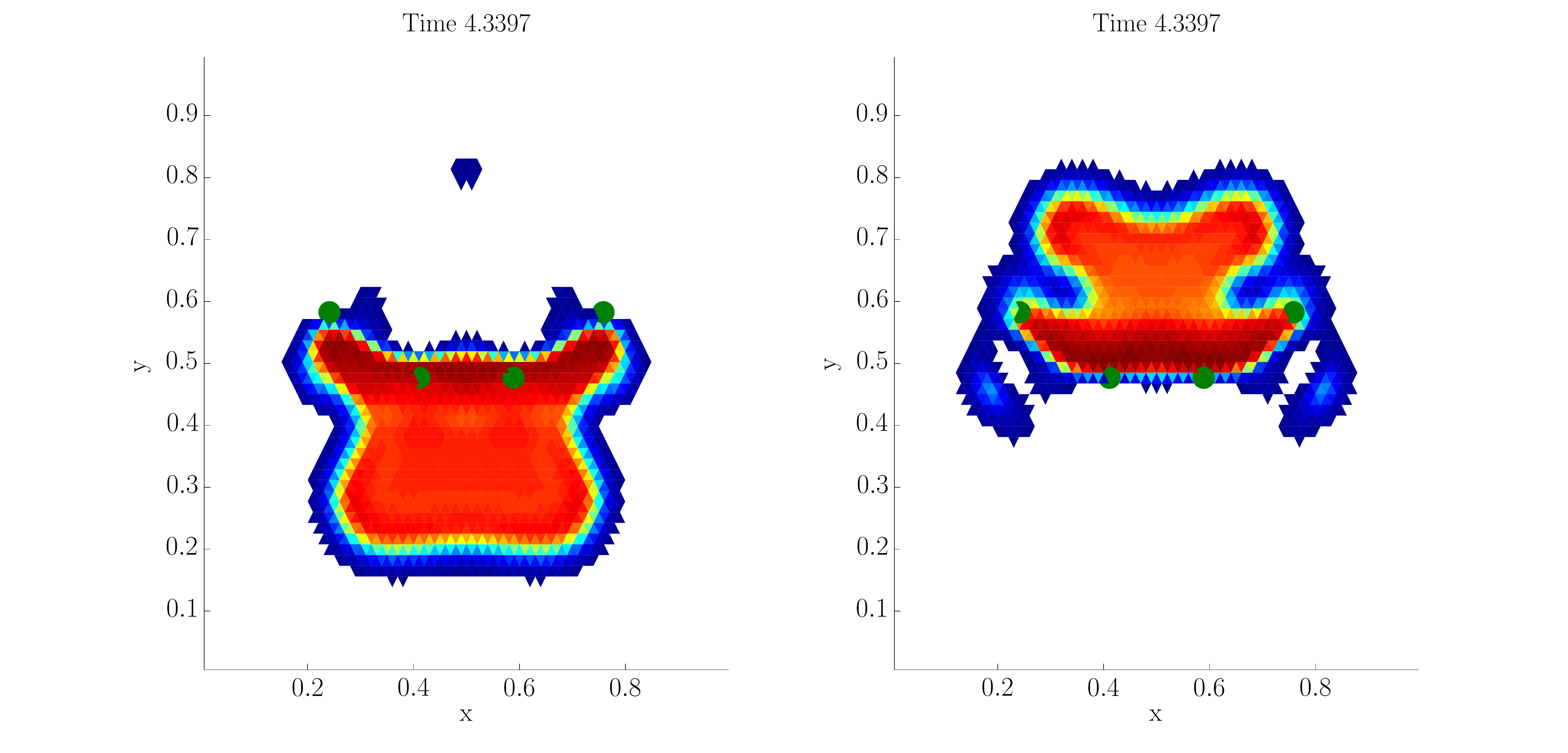}\hfill%
  \includegraphics*[width=0.45\textwidth, viewport=125 30 1350
  800]{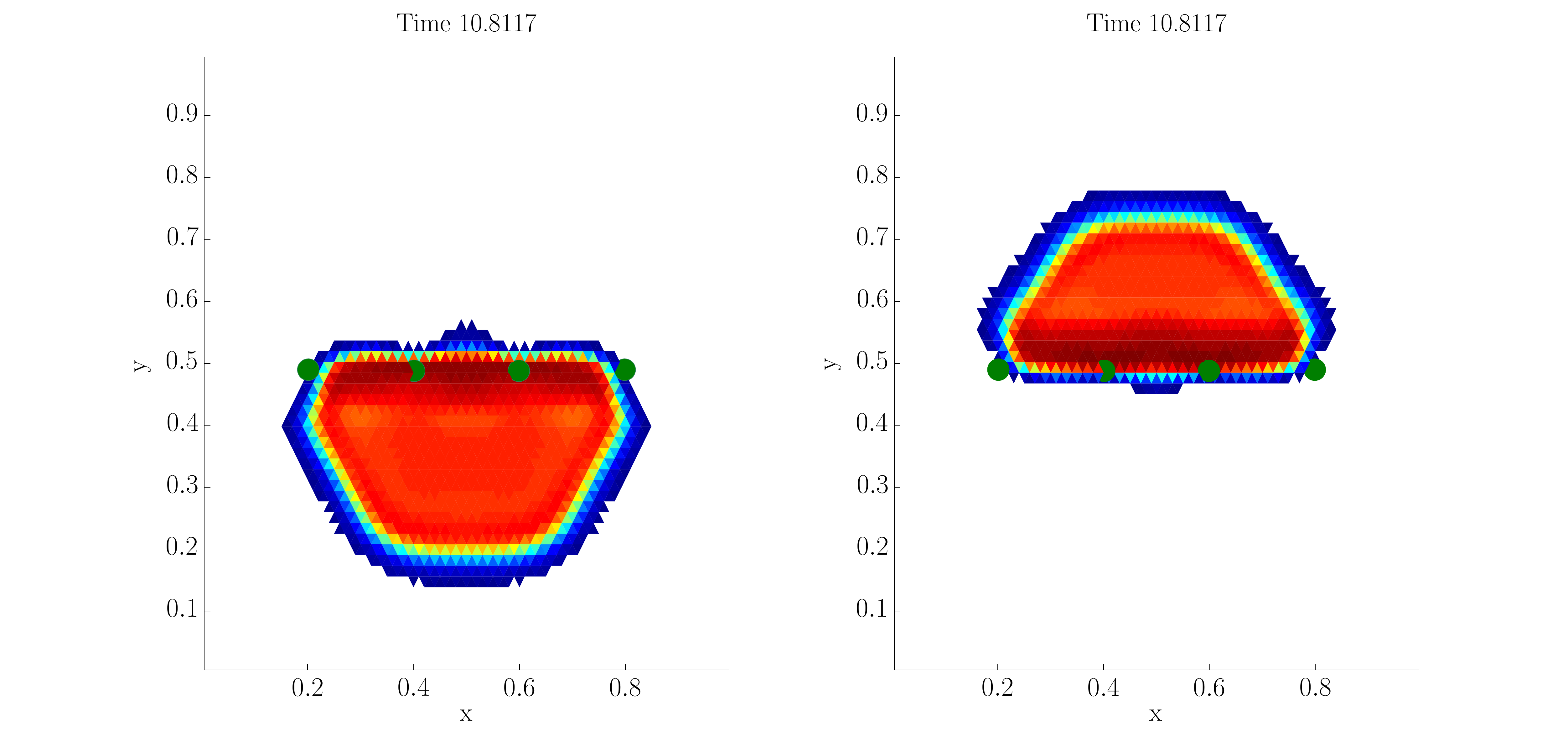}\\
  \caption{Plots of the solution
    to~\eqref{eq:Ex3_1}--\eqref{eq:Ex3_2}--%
    \eqref{eq:Ex3_ID} on the $(x,y)$ plane. In each of the four pairs
    of diagrams, $\rho_1$ is on the left and $\rho_2$ is on the right.
    Above left: two groups of hooligans start fighting. Above right:
    two officers start to separate the fighting groups. Bottom left:
    the officers succeed in separating the groups in the central part,
    but fighting continues on the sides. Bottom right: the four
    officers succeed in separating the two groups.}
  \label{fig:po_int_hools_1}
\end{figure}
At the beginning the two groups of hooligans start fighting in the
middle of the domain, while some part of the groups split from
the rest and stay calm (Figure~\ref{fig:po_int_hools_1}, top left).
As the conflicting groups mix, the police approaches and tries to
separate them.  The first two officers can not completely isolate the
groups (Figure~\ref{fig:po_int_hools_1}, top right) as at the
boundaries the hooligans still attack.  This stops when the other two
policemen join the line of officers (Figure \ref{fig:po_int_hools_1},
bottom left).  At the end the police can separate the conflicting
parties (Figure \ref{fig:po_int_hools_1}, bottom right). This latter
configuration appears to be relatively stationary.
\begin{figure}[ht!]
  \centering
  \includegraphics*[width=0.45\textwidth, viewport=125 30 1350
  800]{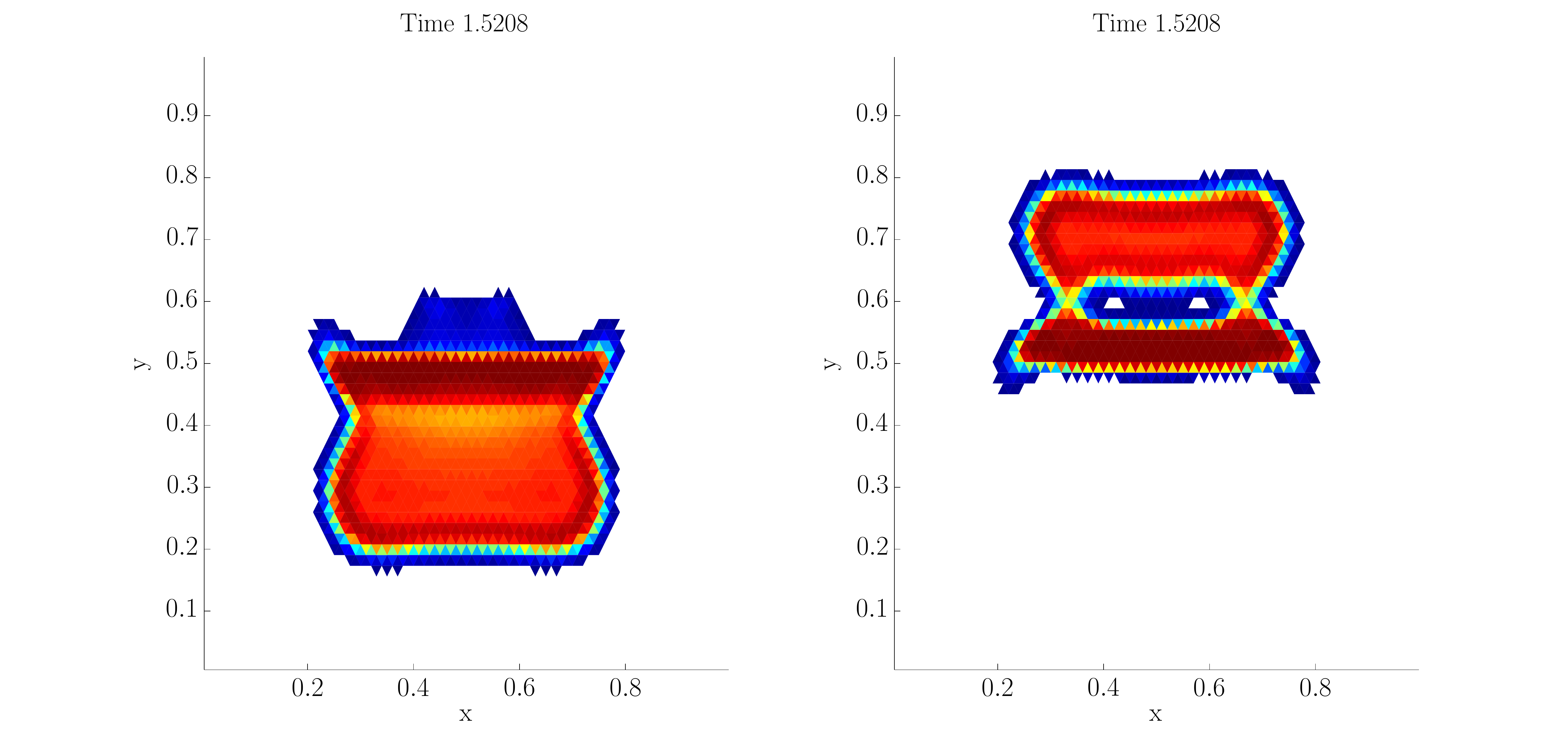}\hfill%
  \includegraphics*[width=0.45\textwidth, viewport=125 30 1350
  800]{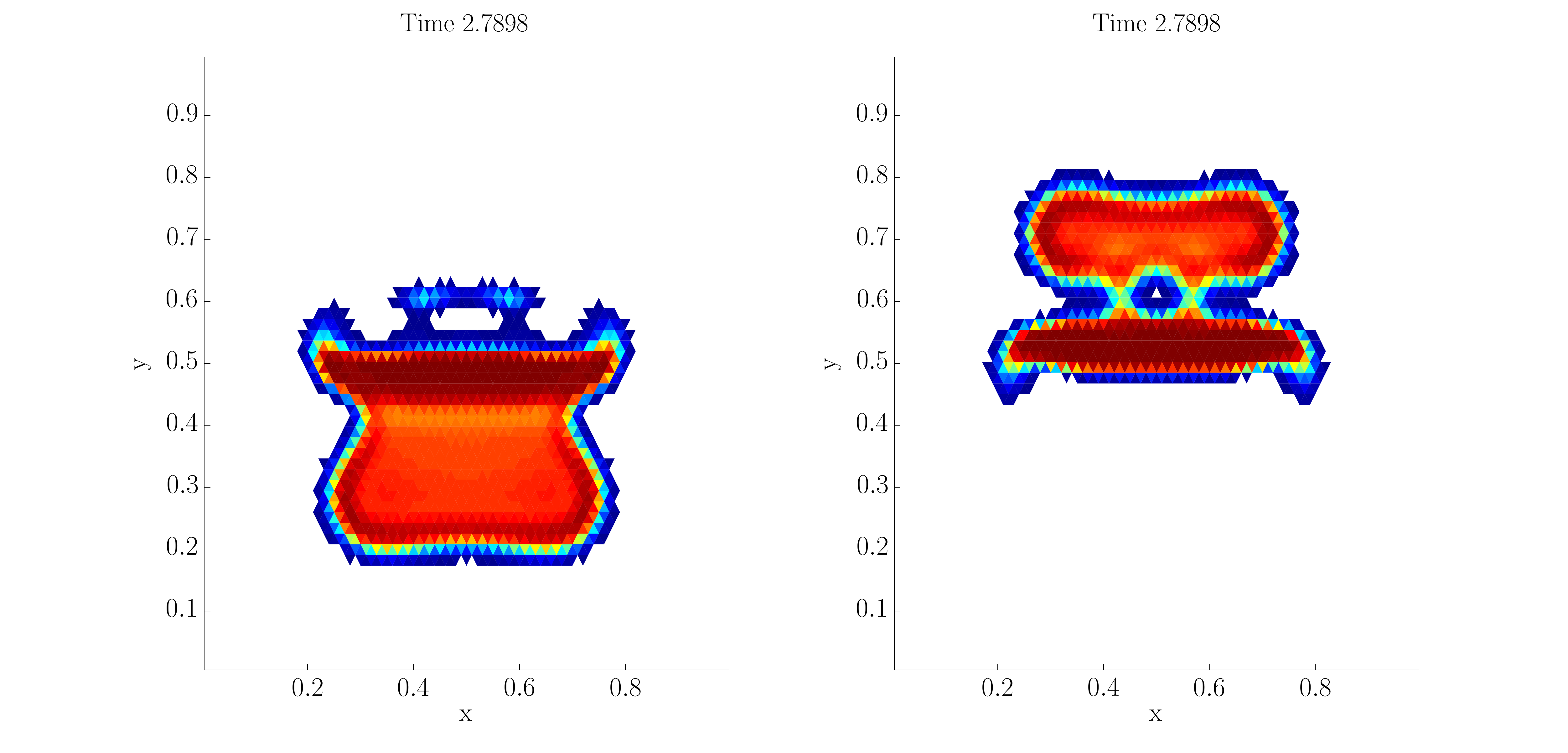}\\
  \includegraphics*[width=0.45\textwidth, viewport=125 30 1350
  800]{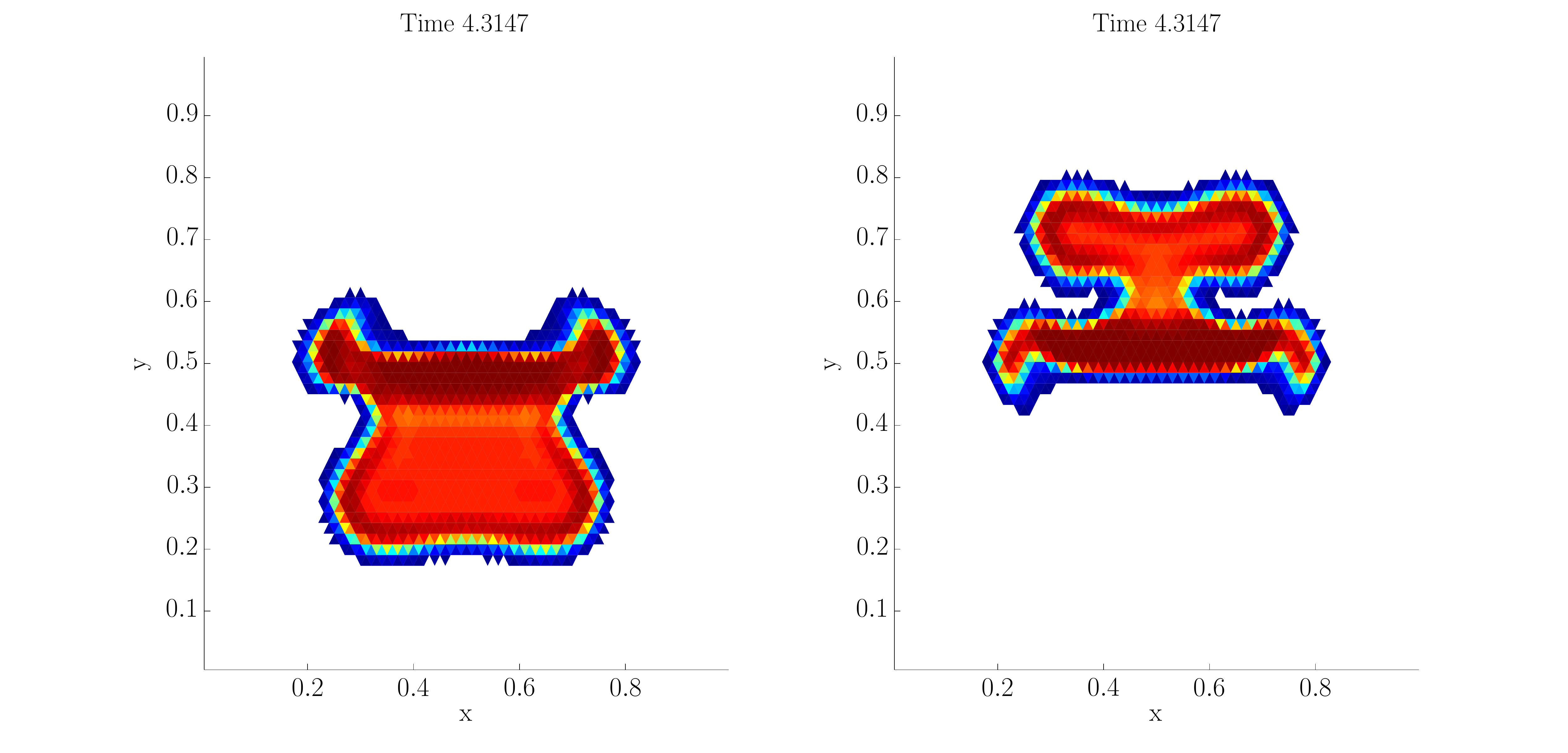}\hfill%
  \includegraphics*[width=0.45\textwidth, viewport=125 30 1350
  800]{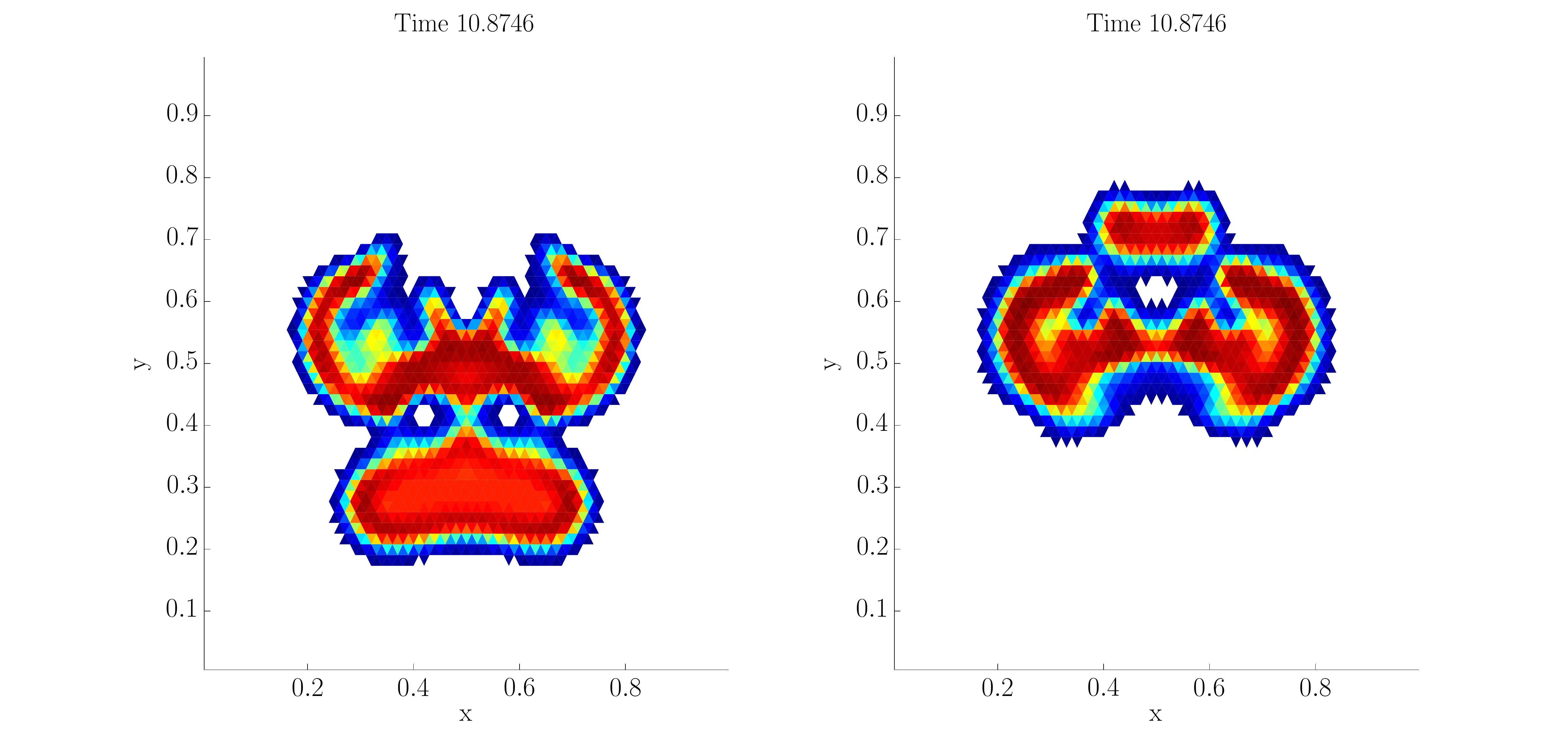}\\
  \caption{Plots of the solution
    to~\eqref{eq:Ex3_1}--\eqref{eq:Ex3_2}--%
    \eqref{eq:Ex3_ID} on the $(x,y)$ plane, but with $\epsilon_3 =
    \epsilon_4 = 0$, so that police officers are absent. Note that,
    differently from the integration shown in
    Figure~\ref{fig:po_int_hools_1}, here the two groups are
    superimposed, meaning that a fight takes place. In each of the
    four pairs of diagrams, $\rho_1$ is on the left and $\rho_2$ is on
    the right.}
  \label{fig:hools_no_police}
\end{figure}
The same equations, but with no police officers so that $\epsilon_3 =
\epsilon_4 = 0$, is displayed in
Figure~\ref{fig:hools_no_police}. Note that the two groups superimpose
and in the region occupied by both a fight takes place.

\section{Technical Details}
\label{sec:technical}

Denote $W_d = \int_0^{\pi/2} \left(\cos(\theta)\right)^d
\d\theta$. For later use, we state here without proof the Gr\"onwall
type lemma used in the sequel.

\begin{lemma}
  \label{lem:Gronwall}
  Let $T > 0$, $\delta \in \C0\left([0,T];\reali^+\right)$, $\alpha
  \in \Lloc{\infty}\left([0,T];\reali^+\right)$ and $\beta \in
  \Lloc1([0,T];\pint\reali^+)$. If $\displaystyle \delta(t) \leq
  \alpha(t) + \int_0^t \beta(\tau) \, \delta(\tau) \d\tau$ for a.e.~$t
  \in [0, T]$ then,
  \begin{displaymath}
    \delta(t)
    \leq \alpha(t) + \int_0^t \alpha(\tau) \, \beta(\tau) \,
    e^{\int_\tau^t \beta(s)\, \d{s}} \, \d\tau \leq \left(\sup_{\tau \in
        [0,t]} \alpha (\tau)\right) \; e^{\int_0^t \beta(\tau)\,
      \d\tau} \,,\quad \text{for a.e.~$t \in [0,T]$.}
  \end{displaymath}
\end{lemma}

\noindent The proof is immediate and hence omitted.

The well posedness of the Cauchy problem
\begin{equation}
  \label{eq:chauchy-1}
  \left\{
    \begin{array}{l}
      \partial_t \rho + \nabla_x \cdot \left(q(\rho) V(t,x)\right) = 0,
      \\
      \rho (0,x) = \rho_o (x).
    \end{array}
  \right.
\end{equation}
follows from~\cite[Proposition~2.9]{magali-improved}.

\begin{lemma}
  \label{lem:stab-1}
  Assume $R>0$ and
  \begin{eqnarray}
    \label{eq:stab-1:q}
    q & \in &\C2(\reali^+; \reali^+) \textrm{ satisfies } q(0) = 0
    \textrm{ and } q(R) = 0,
    \\
    \label{eq:stab-1:V}
    V & \in &\C2 (\reali^+ \times \reali^d; \reali^d)
    \textrm{ satisfies }
    \left\{
      \begin{array}{l}
        \div V(t, \cdot) \in \W{1}{1} (\reali^d; \reali^d),
        \\
        V(t, \cdot) \in \W{1}{\infty} (\reali^d; \reali^d),
      \end{array}
    \right.
    \textrm{ for } t \in \reali^+,
    \\
    \label{eq:stab-1:rho}
    \rho_o & \in &\L1 (\reali^d; [0,R]).
  \end{eqnarray}
  Then, there exists a unique Kru\v zkov solution $\rho \in \C0
  \left(\reali^+; \L1(\reali^d; [0,R])\right)$ to~\eqref{eq:chauchy-1}
  and
  \begin{equation}
    \label{eq:3}
    \norma{\rho (t)}_{\L1} = \norma{\rho_o}_{\L1}
    \qquad \mbox{for all } t \in \reali^+ \,.
  \end{equation}
  If moreover $\rho_o \in \BV (\reali^d; [0,R])$, then, for every $t >
  0$
  \begin{equation}
    \label{eq:stab-1-BV-est}
    \tv \left(\rho(t)\right)
    \leq
    \left(\tv\left(\rho_o\right) + d \, W_d
      \norma{q}_{\L\infty} \int_0^t \int_{\reali^d}
      \norma{\nabla_x \div V(\tau,x)}_{\reali^d} \d{x} \d\tau\right)
    e^{\kappa_o\, t},
  \end{equation}
  and for every $0 < t_1 < t_2$,
  \begin{equation}
    \label{eq:estimate-in-t}
    \begin{array}{rcl}
      \displaystyle
      \norma{\rho(t_2) - \rho(t_1)}_{\L1}
      & \leq &
      \displaystyle
      \norma{q}_{\L\infty}
      \int_{t_1}^{t_2} \int_{\reali^d}
      \modulo{\nabla_x \cdot V(t,x)} \d{x} \d{t}
      \\[16pt]
      & &
      \displaystyle
      +
      (t_2 - t_1) \norma{q'}_{\L\infty} \norma{V}_{\L\infty}
      \sup_{\tau \in [0, t_2]} \tv \left(\rho(\tau)\right),
    \end{array}
  \end{equation}
  where $\kappa_o = (2d + 1) \norma{q'}_{\L\infty} \norma{\nabla_x
    V}_{\L\infty}$.

  Let $q_1$, $q_2$, $V_1$, $V_2$ and $\rho_o^1$, $\rho_o^2$
  satisfy~\eqref{eq:stab-1:q}, \eqref{eq:stab-1:V}
  and~\eqref{eq:stab-1:rho}.  Call $\rho_1$, $\rho_2$ the solutions
  to
  \begin{equation}
    \label{eq:stab-1:2-system}
    \left\{
      \begin{array}{l}
        \partial_t \rho_1 + \div \left(q_1(\rho_1) V_1(t,x)\right) = 0,
        \\
        \rho_1(0,x) = \rho_o^{1}(x),
      \end{array}
    \right.
    \textrm{ and }
    \left\{
      \begin{array}{l}
        \partial_t \rho_2 + \div \left(q_2(\rho_2) V_2(t,x)\right) = 0,
        \\
        \rho_2(0,x) = \rho_o^{2}(x).
      \end{array}
    \right.
  \end{equation}
  Then, for every $t \in \reali^+$,
  \begin{eqnarray*}
    \norma{\rho_1(t) - \rho_2(t)}_{\L1}\!\!\!
    & \leq & \!\!\!
    \norma{\rho_o^{1} - \rho_o^{2}}_{\L1} e^{\kappa t}
    \\
    & + &\!\!\!
    \frac{\kappa_o e^{\kappa_o t} - \kappa e^{\kappa t}}{\kappa_o - \kappa}
    \left[
      \tv(\rho_o^1)
      +
      d\, W_d  \norma{q_1}_{\L\infty}
      \norma{\nabla_x \div V_1}_{\L1 ([0,t];\L1)}
    \right]
    \\
    & &
    \quad \times
    \left[
      \norma{q'_2}_{\L\infty} \norma{V_1 - V_2}_{\L1([0,t];\L\infty)}
      +
      \norma{q'_1 - q'_2}_{\L\infty} \norma{V_1}_{\L1([0,t];\L\infty)}
    \right]
    \\
    & + &\!\!\!\!\!
    \left[
      \norma{q_1}_{\L\infty}
      \norma{\div (V_1 - V_2)}_{\L1 ([0,t];\L1)}
      +
      \norma{q_1-q_2}_{\L\infty}
      \norma{\div V_2}_{\L1 ([0,t];\L1)}
    \right]
    e^{\kappa t}\!,
  \end{eqnarray*}
  where
  \begin{equation}
    \label{eq:kappakappa}
    \begin{array}{rcl}
      \kappa_o
      & = &
      \displaystyle
      (2d + 1) \norma{q'_1}_{\L\infty} \norma{\nabla_x V_1}_{\L\infty([0,t],\L\infty)},
      \\
      \kappa
      & = &
      \displaystyle
      \norma{q_1' \, \div V_1 - q_2' \, \div V_2}_{\L\infty}.
    \end{array}
  \end{equation}
\end{lemma}

\begin{proof}
  The equality~\eqref{eq:3} directly follows from~\textbf{(q)}
  and~\cite[Theorem~1]{Kruzkov}. The total variation
  bound~\eqref{eq:stab-1-BV-est} follows
  from~\cite[Theorem~2.2]{magali-improved}. Estimate~\eqref{eq:estimate-in-t}
  follows from~\cite[Corollary~2.4]{magali-improved}. The stability
  estimate follows from~\cite[Proposition~2.9]{magali-improved}.
\end{proof}

\begin{lemma}
  \label{lem:uniform-p-estimate}
  Assume that~\textbf{(F)} and~\textbf{($\boldsymbol{\mathcal{B}}$)}
  hold. Fix $p_o \in \reali^m$ and $r \in \C0\left(\reali^+;
    \L1(\reali^d; \reali^n)\right)$.  Then, problem
  \begin{displaymath}
    \left\{
      \begin{array}{l}
        \dot p
        =
        F\left(t, p, \left(\mathcal{B}(r)\right) (p)\right),
        \\
        p(0) = p_o,
      \end{array}
    \right.
  \end{displaymath}
  admits a unique Caratheodory solution $p \in \W11 (\reali^+;
  \reali^m)$ and for every $t>0$
  \begin{equation}
    \label{eq:uniform-p-estimate}
    \norma{p(t)}_{\reali^m}
    \leq
    \left(
      \norma{p_o}_{\reali^m}
      +
      \int_0^t C_F(s)
      \left(1 + L_B \, \norma{r (s)}_{\L1} \right)
      \d{s}
    \right)
    \exp \left(\int_0^t C_F(s)\d{s}\right) \,.
  \end{equation}
  If $T > 0$, $p_o^1,p_o^2 \in \reali^m$ and $r_1,r_2 \in
  \C0\left(\reali^+; \L1 (\reali^d; \reali^n)\right)$, calling
  $p_1,p_2$ the solutions to
  \begin{displaymath}
    \left\{
      \begin{array}{l}
        \dot p
        =
        F_1\left(t, p, \left(\mathcal{B}(r_1)\right) (p)\right),
        \\
        p(0) = p_o^1,
      \end{array}
    \right.
    \qquad\qquad
    \left\{
      \begin{array}{l}
        \dot p
        =
        F_2\left(t, p, \left(\mathcal{B}(r_2)\right) (p)\right),
        \\
        p(0) = p_o^2,
      \end{array}
    \right.
  \end{displaymath}
  for every $t \in [0,T]$ the following estimate holds
  \begin{equation}
    \label{eq:p}
    \begin{array}{rcl}
      \norma{p_1 (t) - p_2 (t)}_{\reali^m}
      & \leq &
      \displaystyle
      \left[
        \norma{p_o^1-p_o^2}_{\reali^m}
        +
        t \norma{F_1 - F_2}_{\L\infty}
        +
        L_B \norma{r_1 - r_2}_{\L1 ([0,t];\L1)}
      \right]
      \\
      & &
      \displaystyle
      \qquad \times
      e^{L_F (t + L_B \norma{r_1}_{\L1 ([0,t];\L1)})} .
    \end{array}
  \end{equation}
\end{lemma}

\begin{proof}
  The existence and uniqueness of the solution follows, for instance,
  from~\cite[Theorem~2.1.1]{bressan-piccoli-book}. Moreover,
  by~\textbf{(F)} and~\textbf{($\boldsymbol{\mathcal{B}}$)},
  \begin{eqnarray*}
    \norma{p(t)}_{\reali^m}
    & = &
    \norma{
      p_o
      +
      \int_0^t
      F\left(
        s, p(s), \left(\mathcal{B}(r_1(s))\right) \left(p(s)\right)
      \right)
      \d{s}}_{\reali^m}
    \\
    & \leq &
    \norma{p_o}_{\reali^m}
    +
    \int_0^t C_F(s)
    \left[
      1
      +
      \norma{p(s)}_{\reali^m}
      +
      \norma{\left(\mathcal{B}(r_1(s))\right) \left(p(s)\right)}_{\reali^\ell}
    \right] \d{s}
    \\
    & \leq &
    \norma{p_o}_{\reali^m}
    +
    \int_0^t C_F(s)
    \left(1 + L_B \, \norma{r_1 (s)}_{\L1} \right)
    \d{s}
    +
    \int_0^t C_F(s) \norma{p(s)}_{\reali^m} \d{s}.
  \end{eqnarray*}
  By Lemma~\ref{lem:Gronwall}, we
  deduce~\eqref{eq:uniform-p-estimate}. To prove the stability
  estimate, first note that, given $T>0$,
  by~\eqref{eq:uniform-p-estimate} there exists a compact set $K
  \subseteq \reali^m$ such that $p_1(t), p_2(t) \in K$ for every $t
  \in [0, T]$. Denote with $L_F$ the constant related to $K$
  in~\textbf{(F)}.  Using~\textbf{(F)}
  and~\textbf{($\boldsymbol{\mathcal{B}}$)} we get
  \begin{eqnarray*}
    & &
    \norma{p_1 (t) - p_2 (t)}_{\reali^m}
    \\
    & \leq &
    \norma{p_o^1-p_o^2}_{\reali^m}
    \\
    & &
    +
    \int_0^t
    \norma{
      F_1\left(
        \tau,p_1 (\tau),
        \left(\mathcal{B} (r_1)\right) \left(p_1 (\tau)\right)\right)
      -
      F_2\left(
        \tau,p_1 (\tau),
        \left(\mathcal{B} (r_1)\right) \left(p_1 (\tau)\right)\right)
    }_{\reali^m}
    \d\tau
    \\
    & &
    +
    \int_0^t
    \norma{
      F_2\left(
        \tau,p_1 (\tau),\left(\mathcal{B} (r_1)\right)
        \left(p_1 (\tau)\right)\right)
      -
      F_2\left(
        \tau,p_1 (\tau),\left(\mathcal{B} (r_1)\right)
        \left(p_2 (\tau)\right)\right)
    }_{\reali^m}
    \d\tau
    \\
    & &
    +
    \int_0^t
    \norma{
      F_2\left(
        \tau,p_1 (\tau),\left(\mathcal{B} (r_1)\right)
        \left(p_2 (\tau)\right)\right)
      -
      F_2\left(
        \tau,p_1 (\tau),\left(\mathcal{B} (r_2)\right)
        \left(p_2 (\tau)\right)\right)
    }_{\reali^m}
    \d\tau
    \\
    & &
    +
    \int_0^t
    \norma{
      F_2\left(
        \tau,p_1 (\tau),\left(\mathcal{B} (r_2)\right)
        \left(p_2 (\tau)\right)\right)
      -
      F_2\left(
        \tau,p_2 (\tau),\left(\mathcal{B} (r_2)\right)
        \left(p_2 (\tau)\right)\right)
    }_{\reali^m}
    \d\tau
    \\
    & \leq &
    \norma{p_o^1-p_o^2}_{\reali^m} + t \norma{F_1 - F_2}_{\L\infty}
    \\
    & &
    +
    L_F \int_0^t
    \norma{
      \left(\mathcal{B}\left(r_1 (\tau)\right)\right) \left(p_1 (\tau)\right)
      -
      \left(\mathcal{B}\left(r_1 (\tau)\right)\right) \left(p_2 (\tau)\right)
    }_{\reali^\ell}
    \d\tau
    \\
    & &
    +
    L_F \int_0^t
    \norma{
      \left(\mathcal{B}\left(r_1 (\tau)\right)\right) \left(p_2 (\tau)\right)
      -
      \left(\mathcal{B}\left(r_2 (\tau)\right)\right) \left(p_2 (\tau)\right)
    }_{\reali^\ell}
    \d\tau
    \\
    & &
    +
    L_F \int_0^t
    \norma{p_1 (\tau) - p_2 (\tau)}_{\reali^m}
    \d\tau
    \\
    & \leq &
    \norma{p_o^1-p_o^2}_{\reali^m}
    +    t \norma{F_1 - F_2}_{\L\infty}
    +
    L_B \int_0^t\norma{r_1 (\tau) - r_2 (\tau)}_{\L1} \d\tau
    \\
    & &
    +
    L_F \int_0^t \left(1+ L_B \norma{r_1 (\tau)}_{\L1} \right)
    \norma{p_1 (\tau) - p_2 (\tau)}_{\reali^m}
    \d\tau \,.
  \end{eqnarray*}
  Apply Lemma~\ref{lem:Gronwall} to complete the proof.
\end{proof}

\begin{proofof}{Theorem~\ref{thm:main}}
  The proof is divided in various steps.

  \paragraph{Introduction of $X$ and $\mathcal{T}$.}
  Fix the initial data $(\rho_o,p_o) \in (\L1 \cap \BV) (\reali^d;
  [0,R]^n) \times \reali^m$ and a positive $T < 1$. For positive
  $\Delta_\rho, \Delta_p$, define the closed balls
  \begin{displaymath}
    B_\rho
    =
    \left\{
      \rho \in \L1 (\reali^d; [0,R]^n) \colon
      \norma{\rho - \rho_o}_{\L1} \leq \Delta_\rho
    \right\}
    \quad \mbox{ and } \quad
    B_p
    =
    \left\{
      p \in \reali^m \colon
      \norma{p-p_o}_{\reali^m} \leq \Delta_p
    \right\}
  \end{displaymath}
  and the space
  \begin{displaymath}
    X = \C0 ([0,T]; B_\rho \times B_p)
  \end{displaymath}
  which is a complete metric space with the distance
  \begin{displaymath}
    d \left(\left(\rho_1, p_1\right), \left(\rho_2, p_2\right)\right)
    =
    \sup_{t \in [0,T]}
    \norma{\rho_1(t) - \rho_2(t)}_{\L1}
    +
    \sup_{t \in [0,T]}
    \norma{p_1(t) - p_2(t)}_{\reali^m} .
  \end{displaymath}
  Consider the function $\mathcal{T} \colon X \to X$, with
  $\mathcal{T}(r,\pi) = (\rho, p)$, if $(\rho, p)$ is the solution to
  \begin{equation}
    \label{eq:approx-system}
    \left\{
      \begin{array}{l@{\qquad}l}
        \partial_t \rho^i
        +
        \div
        \left[
          q^i(\rho^i) \;v^i\left(t,x, \mathcal{A}^i(r), \pi\right)
        \right] = 0, &
        i \in \left\{1, \ldots, n\right\},
        \\
        \dot p = F \left(t, p, \left(\mathcal{B}(r)\right) (p)\right),
        \\
        \rho^i(0,x) = \rho_o^i(x), &
        i \in \left\{1, \ldots, n\right\},
        \\
        p(0) = p_o \,.
      \end{array}
    \right.
  \end{equation}
  In the spirit of Definition~\ref{def:solution-1}, here by
  \emph{solution} we mean that $(\rho,p) \in \C0 ([0,T]; \L1
  (\reali^d; \reali^n)\times \reali^m)$ satisfies $(\rho,p) (0) =
  (\rho_o,p_o)$ and for all $i=1, \ldots, n$, the following inequality
  holds
  \begin{equation}
    \label{eq:solRho}
    \begin{array}{@{}l@{}}
      \displaystyle
      \int_0^T \int_{\reali^d}
      \sgn\left(\rho^i (t,x) -k\right)
      \Big[
      \left(\rho^i (t,x) -k\right) \partial_t \phi (t,x)
      \\
      \displaystyle
      \qquad
      +
      \left(q^i \left(\rho^i (t,x)\right) - q^i (k)\right)
      v^i\left(
        t,
        x,
        \left(\mathcal{A}^i\left(r (t)\right)\right) (x),
        \pi (t)
      \right)
      \nabla_x \phi (t,x)
      \Big]
      \d{x} \d{t}
      \geq 0
    \end{array}
  \end{equation}
  for all $\phi \in \Cc1 (\left]0,T\right[\times\reali^d; \reali^+)$
  and for all $k \in \reali$; while, for the $p$ component,
  \begin{equation}
    \label{eq:solP}
    p (t)
    =
    p_o
    +
    \int_0^t
    F\left(
      \tau,
      p (\tau),
      \left(\mathcal{B}\left(r (\tau)\right)\right)\left(p (\tau)\right)
    \right)
    \d\tau
  \end{equation}
  for $t \in [0,T]$.  Lemma~\ref{lem:stab-1} and
  Lemma~\ref{lem:uniform-p-estimate} ensure
  that~\eqref{eq:approx-system} admits a unique solution.

  \paragraph{$\mathcal{T}$ is well defined.}

  To bound the $p$ component, we use~\textbf{(F)},
  \textbf{($\boldsymbol{\mathcal{B}}$)}
  and~\eqref{eq:uniform-p-estimate}
  \begin{eqnarray*}
    & &
    \norma{p (t) - p_o}_{\reali^m}
    \\
    & \leq &
    \int_0^t
    \norma{
      F\left(
        s, p (s), \left(\mathcal{B} \left(r(s)\right) \right) \left(p (s)\right)
      \right)}_{\reali^m}
    \d{s}
    \\
    & \leq &
    \int_0^t C_F (s)
    \left(
      1
      +
      \norma{p (s)}_{\reali^m}
      +
      \norma{
        \left(\mathcal{B} \left(r(s)\right) \right) \left(p (s)\right)
      }_{\reali^\ell}
    \right)
    \d{s}
    \\
    & \leq &
    \int_0^t C_F (s)
    \left(
      1
      +
      \norma{p (s)}_{\reali^m}
      +
      L_B \norma{r (s)}_{\L1}
    \right)
    \d{s}
    \\
    & \leq &
    \left(1 + L_B \left(\norma{\rho_o}_{\L1} + \Delta_\rho \right)\right)
    \int_0^t C_F (s) \d{s}
    \\
    & &
    +
    \int_0^t
    C_F (s)
    \left(
      \norma{p_o}_{\reali^m}
      +
      \left(1 + L_B \left(\norma{\rho_o}_{\L1} + \Delta_\rho \right) \right)
      \int_0^s C_F(\tau) \d{\tau}
    \right)
    e^{\int_0^s C_F(\tau) \d{\tau}}
    \d{s}
  \end{eqnarray*}
  and the latter term above can be made smaller than $\Delta_p$ if $T$
  is sufficiently small.

  To obtain similar estimates for the $\rho$ component, we set $V
  (t,x) = v^i\left(t,x,\mathcal{A}^i (r),\pi\right)$ for $i=1, \ldots,
  n$ and compute % $\div V(\tau,x)$:
  \begin{eqnarray}
    \nonumber
    \div V(t,x)
    & = &
    \sum_{j=1}^d
    \partial_{x_j} v_j^i
    \left(
      t,
      x,
      \left(\mathcal{A}^i \left(r(t)\right)\right)(x),
      \pi(t)
    \right)
    \\
    \nonumber
    & &
    +
    \sum_{j,h=1}^d
    \nabla_{A_h} v_j^i
    \left(
      t,
      x,
      \left(\mathcal{A}^i \left(r(t)\right)\right) (x),
      \pi(t)
    \right)
    \;
    \partial_{x_j} \left(\mathcal{A}^i_h \left(r(t)\right)\right) (x)
    \\
    \nonumber
    & = &
    \nabla_x \cdot
    v^i
    \left(
      t,
      x,
      \left(\mathcal{A}^i \left(r(t)\right)(x)\right),
      \pi(t)
    \right)
    \\
    \label{eq:divV}
    & &
    +
    \partial_{A} v^i
    \left(
      t,
      x,
      \left(\mathcal{A}^i \left(r(t)\right)\right)(x),
      \pi(t)
    \right)
    \cdot
    \nabla_x \left(\mathcal{A}^i\left(r(t)\right)\right)(x),
    \\
    \nonumber
    \nabla_x \div V\left(t,x\right)
    & = &
    \nabla_x \nabla_x
    \cdot v^i
    \left(
      t,
      x,
      \left(\mathcal{A}^i \left(r(t)\right)\right)(x),
      \pi(t)
    \right)
    \\
    \nonumber
    & &
    +
    2\nabla_x \cdot \nabla_A v^i
    \left(
      t,
      x,
      \left(\mathcal{A}^i \left(r(t)\right)\right)(x),
      \pi(t)
    \right)
    \cdot
    \nabla_x \left(\mathcal{A}^i \left(r(t)\right)\right)(x)
    \\
    \nonumber
    & &
    +
    \nabla^2_A v^i\left(
      t,
      x,
      \left(\mathcal{A}^i \left(r(t)\right)\right)(x),
      \pi(t)
    \right)
    \cdot
    \left[\nabla_x
      \left(\mathcal{A}^i \left(r(t)\right)\right) (x)
    \right]^2
    \\
    \nonumber
    & &
    +
    \nabla_{A} v^i
    \left(
      t,
      x,
      \left(\mathcal{A}^i \left(r(t)\right)\right)(x),
      \pi(t)
    \right)
    \cdot
    \nabla^2_x \left(\mathcal{A}^i \left(r(t)\right)\right)(x),
  \end{eqnarray}
  and by~\textbf{(v)} and~\textbf{($\boldsymbol{\mathcal{A}}$)},
  setting $K = \overline{B (p_o, \Delta_p)}$,
  \begin{eqnarray}
    \label{eq:prima}
    \int_0^{t} \int_{\reali^d}
    \modulo{\div V(s,x)} \d{x}\, \d{s}
    & \leq &
    \norma{\mathcal{C}_K}_{\L1} t
    +
    L_A \norma{\mathcal{C}_K}_{\L\infty} \int_0^t \norma{r (s)}_{\L1} \d{s}
    \\
    \label{eq:questa}
    & \leq &
    t
    \left(
      \norma{\mathcal{C}_K}_{\L1}
      +
      L_A \norma{\mathcal{C}_K}_{\L\infty}
      \left(\norma{\rho_o}_{\L1} + \Delta_\rho\right)
    \right),
  \end{eqnarray}
  \begin{eqnarray}
    \nonumber
    & &
    \int_0^t \int_{\reali^d}
    \norma{\nabla_x \div V(\tau,x)}_{\reali^d} \d{x} \d\tau
    \\
    \nonumber
    & \leq &
    \norma{\mathcal{C}_K}_{\L1} t
    +
    2 L_A t \norma{\mathcal{C}_K}_{\L\infty}
    \left(\norma{\rho_o}_{\L1} + \Delta_\rho\right)
    +
    L_A^2 t \norma{\mathcal{C}_K}_{\L\infty}
    \left( \norma{\rho_o}_{\L1} + \Delta_\rho\right)^2
    \\
    & &
    + \nonumber
    L_A t \norma{\mathcal{C}_K}_{\L\infty}
    \left(\norma{\rho_o}_{\L1} + \Delta_\rho\right)
    \\
    \label{eq:quella}
    & = &
    \left(
      \norma{\mathcal{C}_K}_{\L1}
      +
      3 L_A \norma{\mathcal{C}_K}_{\L\infty}
      \left(\norma{\rho_o}_{\L1} + \Delta_\rho\right)
      +
      L_A^2 \norma{\mathcal{C}_K}_{\L\infty}
      \left( \norma{\rho_o}_{\L1} + \Delta_\rho \right)^2
    \right) t.
  \end{eqnarray}
  Proceeding to the $\rho$ component, using~\eqref{eq:estimate-in-t}
  and~\eqref{eq:stab-1-BV-est} together with~\eqref{eq:questa}
  and~\eqref{eq:quella} above
  \begin{eqnarray*}
    & & \norma{\rho (t) - \rho_o}_{\L1}
    \\
    & \leq &
    \int_0^{t} \int_{\reali^d}
    \norma{q}_{\L\infty}
    \modulo{\div V(s,x)} \d{x}\, \d{s}
    +
    t \norma{q'}_{\L\infty} \norma{V}_{\L\infty}
    \sup_{\tau \in [0, t]} \tv \left(\rho(\tau, \cdot)\right)
    \\
    & \leq &
    \norma{q}_{\L\infty}
    \int_0^{t} \int_{\reali^d}
    \modulo{\div V(s,x)} \d{x}\, \d{s}
    \\
    & &
    +
    t \norma{q'}_{\L\infty} \norma{V}_{\L\infty}
    \left(\tv\left(\rho_o\right) + d \, W_d
      \norma{q}_{\L{\infty}\left([0,R]\right)} \int_0^t \int_{\reali^d}
      \norma{\nabla_x \div V(\tau,x)}_{\reali^d} \d{x} \d\tau\right)
    e^{\kappa_o\, t}
    \\
    & \leq &
    t \, \norma{q}_{\L\infty}
    \left(
      \norma{\mathcal{C}_K}_{\L1}
      +
      L_A \norma{\mathcal{C}_K}_{\L\infty}
      \left(\norma{\rho_o}_{\L1} + \Delta_\rho\right)
    \right)
    +
    t \norma{\mathcal{C}_K}_{\L\infty} \norma{q'}_{\L\infty}
    \tv (\rho_o) e^{\kappa_o t}
    \\
    & &
    +
    t^2 \norma{\mathcal{C}_K}_{\L\infty} \norma{q'}_{\L\infty}
    d W_d \norma{q}_{\L\infty}
    \\
    &&
    \times
    \left(
      \norma{\mathcal{C}_K}_{\L1}
      +
      3 L_A \norma{\mathcal{C}_K}_{\L\infty}
      \left(\norma{\rho_o}_{\L1} + \Delta_\rho\right)
      +
      L_A^2 \norma{\mathcal{C}_K}_{\L\infty}
      \left( \norma{\rho_o}_{\L1} + \Delta_\rho\right)^2
    \right)
    e^{\kappa_o t},
  \end{eqnarray*}
  where
  \begin{eqnarray}
    \nonumber
    \kappa_o
    & = &
    (2d + 1) \norma{q'}_{\L\infty}
    \norma{\nabla_x V}_{\L\infty}
    \\
    \nonumber
    & = &
    (2d + 1) \norma{q'}_{\L\infty}
    \norma{
      \nabla_x v^i
      +
      \nabla_A v^i \cdot \nabla_x \mathcal{A}^i\left(r(t)\right)
    }_{\L\infty}
    \\
    \label{eq:PrimaDikappat}
    & \leq &
    (2d + 1) \norma{q'}_{\L\infty}
    \left[
      \norma{\mathcal{C}_K}_{\L\infty}
      + \norma{\mathcal{C}_K}_{\L\infty} L_A \norma{r(t)}_{\L1}
    \right]
    \\
    \label{eq:kappat}
    & \leq &
    (2d + 1) \norma{\mathcal{C}_K}_{\L\infty} \norma{q'}_{\L\infty}
    \left[1 + L_A \left(\norma{\rho_o}_{\L1} + \Delta_\rho\right)\right].
  \end{eqnarray}
  Hence, for $T$ sufficiently small, also $\norma{\rho
    (t)-\rho_o}_{\L1} \leq \Delta_\rho$ completing the proof that
  $\mathcal{T}$ is well defined.

  \paragraph{Notation.} In the sequel, for notational simplicity, we
  introduce the Landau symbol $\O$ to denote a bounded quantity,
  possibly dependent on $T$ and on the constants in~\textbf{(v)},
  \textbf{(F)}, \textbf{($\boldsymbol{\mathcal{A}}$)},
  \textbf{($\boldsymbol{\mathcal{B}}$)} and~\textbf{(q)}.

  \paragraph{$\mathcal{T}$ is a contraction.}
  Fix $(r_1, \pi_1), (r_2, \pi_2) \in X$ and denote $(\rho_i, p_i) =
  \mathcal{T} (r_i, \pi_i)$. We now estimate
  $d\left((\rho_1,p_1),(\rho_2,p_2)\right)$. Consider first the $p$
  component. Using Lemma~\ref{lem:uniform-p-estimate} we get
  \begin{eqnarray*}
    \norma{p_1 (t) - p_2 (t)}_{\reali^m}
    & \leq &
    L_B \, \norma{r_1 - r_2}_{\L1 ([0,t];\L1)} \,
    e^{L_F (t + L_B \norma{r_1}_{\L1 ([0,t];\L1)})}
    \\
    & \leq &
    L_B \, t \, \norma{r_1 - r_2}_{\C0 ([0,t];\L1)} \,
    e^{L_F t (1 + L_B (\norma{\rho_o}_{\L1} + \Delta_\rho))}
    \\
    & = &
    \O\,t \, \norma{r_1 - r_2}_{\C0 ([0,t];\L1)}.
  \end{eqnarray*}
  Apply now Lemma~\ref{lem:stab-1} with $\rho_o^1 = \rho_o =
  \rho_o^2$, $q_1 = q = q_2$, $V_j^i (t,x) = v^i\left(t, x,
    \mathcal{A}^i(r_j(t)(x), \pi_j(t)\right)$ for $i = 1, \ldots, n$
  and $j=1,2$. Equality~\eqref{eq:divV} and~\textbf{(v)} allow to
  bound $\kappa$ in~\eqref{eq:kappakappa} as follows
  \begin{displaymath}
    \kappa
    \leq
    \norma{q'}_{\L\infty} \norma{\div (V^i_1-V^i_2)}_{\L\infty}
    \leq
    2 \, \norma{\mathcal{C}_K}_{\L\infty} \, \norma{q'}_{\L\infty}
    \left(1 + L_A\left(\norma{\rho_o}_{\L1} + \Delta_\rho\right)\right),
  \end{displaymath}
  which ensures, together with~\eqref{eq:kappat}
  \begin{eqnarray*}
    \frac{\kappa_o e^{\kappa_o t} - \kappa e^{\kappa t}}{\kappa_o - \kappa}
    & \leq &
    \left(
      1
      +
      (2d + 1) \norma{\mathcal{C}_K}_{\L\infty} \norma{q'}_{\L\infty}
      \left[1 + L_A \left(\norma{\rho_o}_{\L1} + \Delta_\rho\right)\right]t
    \right)
    \\
    & & \times
    \exp\left(
      (2d + 1) \norma{\mathcal{C}_K}_{\L\infty} \norma{q'}_{\L\infty}
      \left[1 + L_A \left(\norma{\rho_o}_{\L1} + \Delta_\rho\right)\right]
      t
    \right)
    \\
    & = &
    \O \,.
  \end{eqnarray*}
  Using also~\eqref{eq:quella} we obtain
  \begin{eqnarray*}
    \!\!\!\!\!\!
    & &
    \!\!\!
    \norma{\rho_1^i(t) - \rho_2^i(t)}_{\L1}
    \\
    \!\!\!\!\!\!
    & \leq &
    \!\!\!
    t \frac{\kappa_o e^{\kappa_o t} - \kappa e^{\kappa t}}{\kappa_o - \kappa} \!
    \left[
      \tv(\rho_o)
      +
      d\, W_d  \norma{q}_{\L\infty}
      \norma{\nabla_x \div V_1^i}_{\L1 ([0,t];\L1)}
    \right]\!\!
    \norma{q'}_{\L\infty} \norma{V_1^i - V_2^i}_{\C0([0,t];\L\infty)}
    \\
    \!\!\!\!\!\!
    & &
    \!\!\!
    +
    \norma{q}_{\L\infty}
    \norma{\div (V_1^i - V_2^i)}_{\L1 ([0,t];\L1)}
    e^{\kappa t}
    \\
    \!\!\!\!\!\!
    & = &
    \!\!\!
    t \O
    \left[
      1
      +
      \norma{\nabla_x \div V_1^i}_{\L1 ([0,t];\L1)}
    \right]
    \norma{V_1^i - V_2^i}_{\C0([0,t];\L\infty)}
    +
    \O
    \norma{\div (V_1^i - V_2^i)}_{\L1 ([0,t];\L1)}.
  \end{eqnarray*}
  By~\eqref{eq:quella}, we get $\norma{\nabla_x \div V_1^i}_{\L1 ([0,t];\L1)}
  = \O$. By~\textbf{(v)} and~\textbf{($\boldsymbol{\mathcal{A}}$)}, we
  bound $\norma{V_1^i - V_2^i }_{\L\infty}$
  \begin{eqnarray}
    \nonumber
    \norma{V_1^i - V_2^i }_{\L\infty}
    \!\!\! \!\!\!& \leq & \!\!\!
    \norma{\mathcal{C}_K}_{\L\infty}
    \sup_{\tau \in[0,t],x \in \reali^d}
    \left[
      \norma{
        \mathcal{A}^i(r_1(\tau))(x)
        -
        \mathcal{A}^i(r_2(\tau))(x)
      }_{\reali^{d}}
      +
      \norma{\pi_1(\tau) -\pi_2(\tau)}_{\reali^m}
    \right]
    \\
    \label{eq:VmenoV}
    & \leq &
    L_A t \norma{\mathcal{C}_K}_{\L\infty} \norma{r_1 - r_2}_{\C0([0,t];\L1)}
    +
    \norma{\mathcal{C}_K}_{\L\infty} \norma{\pi_1 - \pi_2}_{\C0}
    \\
    \nonumber
    & = &
    \O
    \left(t \norma{r_1 - r_2}_{\C0([0,t];\L1)} + \norma{\pi_1 - \pi_2}_{\C0}\right) .
  \end{eqnarray}
  To estimate $\norma{\div (V_1^i - V_2^i)}_{\L1 ([0,t];\L1)}$ we
  first deal with $\modulo{\div \left(V_1^i(\tau,x) -
      V_2^i(\tau,x)\right)}$, which can be estimated by~\textbf{(v)}
  and~\textbf{($\boldsymbol{\mathcal{A}}$)}
  \begin{eqnarray*}
    & &
    \modulo{\div \left(V_1^i(\tau,x) - V_2^i(\tau,x)\right)}
    \\
    & \leq&
    \modulo{\nabla_x \cdot v^i\left(\tau, x, \mathcal{A}^i
        \left(r_1(\tau)\right)(x), \pi_1(\tau)\right)
      -
      \nabla_x \cdot v^i\left(\tau, x, \mathcal{A}^i \left(r_2(\tau)\right)(x), \pi_2(\tau)\right)}
    \\
    & &
    +
    \left| \nabla_A v^i\left(\tau, x, \mathcal{A}^i
        \left(r_1(\tau)\right)(x), \pi_1(\tau)\right) \nabla_x
      \mathcal{A}^i \left(r_1(\tau)\right)(x)\right.
    \\
    & &
    \qquad
    \left.  -
      \nabla_{A} v^i\left(\tau, x, \mathcal{A}^i \left(r_2(\tau)\right)(x), \pi_2(\tau)\right) \nabla_x \mathcal
      A^i\left(r_2(\tau)\right)(x) \right|
    \\
    & \leq &
    \norma{\mathcal{C}_K}_{\L\infty}
    \norma{
      \mathcal{A}^i \left(r_1(\tau)\right)(x)
      -
      \mathcal{A}^i \left(r_2(\tau)\right)(x)}_{\reali^d}
    +
    \modulo{\mathcal C_K(x)}\,
    \norma{\pi_1(\tau) -  \pi_2(\tau)}_{\reali^m}
    \\
    & &
    +
    \norma{\mathcal{C}_K}_{\L\infty} \norma{
      \nabla_x \mathcal{A}^i
      (r_1(\tau))(x) - \nabla_x \mathcal{A}^i (r_2(\tau))(x)
    }_{\reali^{2d}}
    \\
    & &
    +
    \norma{\mathcal{C}_K}_{\L\infty}
    \norma{\nabla_x \mathcal{A}^i\left(r_2(\tau)\right)(x)}_{\reali^{d\times d}}
    \,
    \norma{\nabla_x \mathcal{A}^i
      (r_1(\tau))(x) - \mathcal{A}^i (r_2(\tau))(x)
    }_{\reali^{d\times d}}
    \\
    & &
    +
    \norma{\mathcal{C}_K}_{\L\infty}
    \norma{\nabla_x \mathcal{A}^i\left(r_2(\tau)\right)(x)} _{\reali^{d\times d}}
    \, \norma{\pi_1(\tau) - \pi_2(\tau)}_{\reali^m} \,.
  \end{eqnarray*}
  Therefore, using~\textbf{($\boldsymbol{\mathcal{A}}$)},
  \begin{eqnarray}
    \nonumber
    & &
    \norma{\div (V_1^i - V_2^i)}_{\L1([0,t];\L1)}
    \\
    \nonumber
    & \le &
    L_A \norma{\mathcal{C}_K}_{\L\infty}
    \int_0^t \norma{r_1(\tau) - r_2(\tau)}_{\L1} \d\tau
    +
    \norma{\mathcal{C}_K}_{\L1}
    \int_0^t \norma{\pi_1(\tau) - \pi_2(\tau)}_{\reali^m}  \d\tau
    \\
    & &
    + L_A \norma{\mathcal{C}_K}_{\L\infty}
    \int_0^t \norma{r_1(\tau) - r_2(\tau)}_{\L1} \d\tau
    +
    L_A^2 \norma{\mathcal{C}_K}_{\L\infty}
    \int_0^t \norma{r_2(\tau)}_{\L1}
    \norma{r_1(\tau) - r_2(\tau)}_{\L1} \d\tau
    \nonumber
    \\
    \nonumber
    & &
    +
    L_A \norma{\mathcal{C}_K}_{\L\infty}
    \int_0^t \norma{r_2(\tau)}_{\L1}
    \norma{\pi_1(\tau) - \pi_2(\tau)}_{\reali^m} \, \d\tau
    \\
    \nonumber
    & \leq&
    L_A \norma{\mathcal{C}_K}_{\L\infty}
    \left(2 + L_A\left(\norma{\rho_o}_{\L1}
        + \Delta_\rho\right)\right)
    \int_0^t \norma{r_1(\tau) - r_2(\tau)}_{\L1} \d\tau
    \\
    & &
    +
    \left(\norma{\mathcal{C}_K}_{\L1}
      + L_A \norma{\mathcal C_K}_{\L\infty}\left(\norma{\rho_o}_{\L1}
        + \Delta_\rho\right)\right)
    \int_0^t \norma{\pi_1(\tau) - \pi_2(\tau)}_{\reali^m}  \d\tau
    \label{eq:DeltaDiv}
    \\
    \nonumber
    & \le &
    t \, \O \left(
      \norma{r_1 - r_2}_{\C0([0,t];\L1)}
      +
      \norma{\pi_1 - \pi_2}_{\C0}
    \right).
  \end{eqnarray}
  Going back to the $\rho$ components,
  \begin{eqnarray*}
    \!\!\!\!\!\!& &
    \norma{\rho_1^i(t) - \rho_2^i(t)}_{\L1}
    \\
    \!\!\!\!\!\!& \leq &
    t \O \!
    \left[
      1
      +
      \norma{\nabla_x \div V_1^i}_{\L1 ([0,t];\L1)}
    \right] \!
    \norma{V_1^i - V_2^i}_{\C0([0,t];\L\infty)}
    +
    \O
    \norma{\div (V_1^i - V_2^i)}_{\L1 ([0,t];\L1)}
    \\
    \!\!\!\!\!\!& \leq &
    t \, \O \left(
      \norma{r_1 - r_2}_{\C0([0,t];\L1)}
      +
      \norma{\pi_1 - \pi_2}_{\C0}
    \right).
  \end{eqnarray*}
  The above estimate ensures that for $T$ sufficiently small,
  $\mathcal{T}$ is a contraction. Hence, it admits a unique fixed
  point $(\rho_*, p_*)$, defined on $[0,T]$.

  \paragraph{$(\rho_*,p_*)$ is a solution to~\eqref{eq:1} on $[0,T]$.}
  Writing that $(\rho_*,p_*)$ is a fixed point for $\mathcal{T}$
  in~\eqref{eq:solRho} and~\eqref{eq:solP} shows that $(\rho_*, p_*)$
  solves~\eqref{eq:1} in the sense of Definition~\ref{def:solution-1}
  on $[0,T]$.

  \paragraph{Global uniqueness.}  Consider two solutions
  \begin{displaymath}
    \left(\rho_j, p_j\right)
    \in
    \C0 \left([0,T_j];
      \L1 (\reali^d; [0,R]^n) \times \reali^m
    \right),
    \quad \mbox{ for } j = 1,2,
  \end{displaymath}
  to~\eqref{eq:1} in the sense of Definition~\ref{def:solution-1},
  corresponding to the same initial datum $(\rho_o,p_o) \in (\L1
    \cap \BV) (\reali^d; [0,R]^n) \times \reali^m$.  Define
  \begin{equation*}
    T^\ast = \sup\left\{t \in [0,\min\{T_1,T_2\}]:\, \left(\rho_1, p_1\right)(s)
      = \left(\rho_2, p_2\right)(s)\textrm{ for all } s \in [0,t]\right\}.
  \end{equation*}
  Clearly, $T^\ast \ge 0$ and $\left(\rho_1, p_1\right)(T^\ast) =
  \left(\rho_2, p_2\right)(T^\ast)$.

  Assume by contradiction that $T^\ast < \min\{T_1,T_2\}$ and define
  $\left(\rho^\ast, p^\ast\right) = \left(\rho_1, p_1\right)(T^\ast) =
  \left(\rho_2, p_2\right)(T^\ast)$.  The previous steps, which can be
  applied thanks to the bound~\eqref{eq:stab-1-BV-est}, ensure the
  local existence to problem~\eqref{eq:1} with datum $(\rho^\ast,
  p^\ast)$ assigned at time $T^*$. Hence, $(\rho_1, p_1)(t) = (\rho_2,
  p_2)(t)$ on a full neighborhood of $T^*$. This contradicts the
  assumption $T^\ast < \min\{T_1, T_2\}$, proving global uniqueness.

  \paragraph{$\BV$ estimate on $\rho$ and $\L\infty$ estimate on $p$.}
  Let $(\rho,p)$ be the solution
  to~\eqref{eq:1}. By~\eqref{eq:uniform-p-estimate} and~\eqref{eq:3},
  \begin{equation}
    \label{eq:pInfty}
    \norma{p(t)}_{\reali^m}
    \leq
    \left(
      \norma{p_o}_{\reali^m}
      +
      \left(1 + L_B \, \norma{\rho_o}_{\L1} \right)
      \int_0^t C_F(s) \d{s}
    \right)
    \exp \int_0^t C_F(s)\d{s} \,.
  \end{equation}
  Call $K_t$ the closed ball in $\reali^m$ with radius $ \left[
    \norma{p_o}_{\reali^m} + \left(1 + L_B \, \norma{\rho_o}_{\L1}
    \right) \int_0^t C_F(s) \d{s} \right] e^{\int_0^t
    C_F(s)\d{s}}$. By~\eqref{eq:stab-1-BV-est} and observing that
  $\norma{\rho(t)}_{\L1} = \norma{\rho_o}_{\L1}$ for every $t > 0$
  \begin{eqnarray*}
    \tv \left(\rho(t)\right)
    & \leq &
    \tv\left(\rho_o\right)  e^{\kappa_t\, t}
    \\
    & + &
    d \, W_d
    \norma{q}_{\L\infty}
    t
    \left(
      \norma{\mathcal{C}_{K_t}}_{\L1}
      +
      3 L_A \norma{\mathcal{C}_{K_t}}_{\L\infty} \norma{\rho_o}_{\L1}
      +
      L_A^2 \norma{\mathcal{C}_{K_t}}_{\L\infty}
      \left( \norma{\rho_o}_{\L1} \right)^2
    \right)
    e^{\kappa_t t},
  \end{eqnarray*}
  where, by~\eqref{eq:kappat},
  \begin{displaymath}
    \kappa_t
    =
    (2d+1) \norma{\mathcal{C}_{K_t}}_{\L\infty} \norma{q'}_{\L\infty}
    (1+ L_A \norma{\rho_o}_{\L1}) \,.
  \end{displaymath}
  Hence, $\rho (t) \in \BV (\reali^d; [0,R]^n)$ as long as $\rho$ is
  defined.

  \paragraph{$\rho$ is Lipschitz continuous in time.} Let $(\rho,p)$
  be the solution to~\eqref{eq:1}. Fix $t_1,t_2$ with $t_1 < t_2$
  inside the time interval where $(\rho,p)$ is
  defined. Use~\eqref{eq:estimate-in-t}, \eqref{eq:prima}
  and~\textbf{(v)} to obtain
  \begin{eqnarray*}
    \norma{\rho (t_2) - \rho (t_1)}_{\L1}
    & \leq &
    \displaystyle
    (t_2-t_1) \, \norma{q}_{\L\infty}
    \left(
      \norma{\mathcal{C}_{K_t}}_{\L1}
      +
      L_A \norma{\mathcal{C}_{K_t}}_{\L\infty} \norma{\rho_o}_{\L1}
    \right)
    \\[16pt]
    & &
    \displaystyle
    +
    (t_2 - t_1) \norma{\mathcal{C}_{K_t}}_{\L\infty} \norma{q'}_{\L\infty} \,
    \sup_{\tau \in [0, t_2]} \tv \left(\rho(\tau)\right) \,.
  \end{eqnarray*}
  This estimate, together with the $\BV$ bound above, ensures that on
  any bounded time interval on which it is defined, $\rho$ is
  Lipschitz continuous in time.

  \paragraph{$p$ is uniformly continuous in time.} Let $(\rho,p)$ be
  the solution to~\eqref{eq:1}. Fix $t_1,t_2$ with $t_1 < t_2$ inside
  the time interval where $(\rho,p)$ is
  defined. By~5.~in~\textbf{(F)}, \eqref{eq:pInfty},
  \textbf{($\boldsymbol{\mathcal{B}}$)} and~\eqref{eq:3} we have
  \begin{eqnarray*}
    \!\!\!\!\!\!
    & &
    \norma{p (t_2) - p (t_1)}_{\reali^m}
    \\
    \!\!\!\!\!\!
    & \leq &
    \int_{t_1}^{t_2}
    \norma{
      F\left(
        \tau,
        p (\tau),
        \left(\mathcal{B}\left(\rho (\tau)\right)\right)
        \left(p (\tau)\right)
      \right)}_{\reali^m}
    \d\tau
    \\
    \!\!\!\!\!\!
    & \leq &
    \int_{t_1}^{t_2}
    C_F (\tau)
    \left(1
      +
      \norma{p (\tau)}
      +
      L_B \norma{\rho_o}_{\L1}
    \right)
    \d\tau
    \\
    \!\!\!\!\!\!
    & \leq &
    \int_{t_1}^{t_2}
    C_F (\tau)
    \d\tau
    \left[
      1
      +
      \left(
        \norma{p_o}_{\reali^m}
        +
        \left(1 + L_B \, \norma{\rho_o}_{\L1} \right)
        \int_0^{t_2} C_F(s) \d{s}
      \right)
      e^{\int_0^{t_2} C_F(s)\d{s}}
      +
      L_B \norma{\rho_o}_{\L1}
    \right],
  \end{eqnarray*}
  which shows the uniform continuity of $p$ on any bounded time
  interval.

  \paragraph{Global existence.}  Fix the initial datum $(\rho_o,p_o)
  \in (\L1 \cap \BV) (\reali^d; [0,R]^n) \times \reali^m$.  Define
  \begin{equation*}
    T_* = \sup \left\{
      T > 0 \colon
      \begin{array}{l}
        \exists
        (\rho, p) \in
        \C0 \left([0,T]; \L1 (\reali^d; [0,R]^n) \times \reali^m \right)
        \\
        \left(\rho, p\right)
        \mbox{ solves~\eqref{eq:1} according to Definition~\ref{def:solution-1}}
        \\
        \text{with~} \left(\rho, p\right)(0) = \left(\rho_o, p_o\right)
      \end{array}
    \right\} \,.
  \end{equation*}
  Assume by contradiction that $T_\ast < + \infty$. Then, by the
  existence and uniqueness proved above, there exists a solution
  $(\rho_*,p_*)$ to~\eqref{eq:1} with $(\rho_*,p_*)(0) = (\rho_o,p_o)$
  which is defined on $\left[0, T_*\right[$. By the previous steps,
  the map $t \to (\rho_*,p_*) (t)$ is uniformly continuous on
  $\left[0, T_*\right[$, hence it can be uniquely extended by
  continuity to $[0,T_*]$. Call $(\bar\rho,\bar p) = (\rho_*,p_*)
  (T_*)$. The Cauchy problem consisting of~\eqref{eq:1} with initial
  datum $(\bar\rho, \bar p)$ assigned at time $T_*$ still satisfies
  all conditions to have a unique solution defined also on a right
  neighborhood of $T_*$, which contradicts the choice of $T_*$.

  \paragraph{Continuous dependence from the initial datum.}
  Fix a positive $T$.  For $j = 1, 2$, choose $(\rho_o^j,p_o^j) \in
  (\L1\cap\BV) (\reali^d;[0,R]^m)$ and call $(\rho_j,p_j)$ the
  corresponding solution as constructed above. For any $t \in [0,T]$,
  by~\eqref{eq:p}
  \begin{eqnarray}
    \nonumber
    \norma{p_1(t) - p_2 (t)}_{\reali^m}
    & \leq &
    \left(
      \norma{p_o^1-p_o^2}_{\reali^m}
      +
      L_B \, \norma{\rho_1 - \rho_2}_{\L1 ([0,t];\L1)}
    \right)
    e^{L_F t(1 + L_B \norma{\rho_o^1}_{\L1})}
    \\
    \label{eq:pp}
    & \leq &
    \left(
      \norma{p_o^1-p_o^2}_{\reali^m}
      \!\!\!+
      \O
      \norma{\rho_1 - \rho_2}_{\L1 ([0,t];\L1)}
    \right)
    e^{\O t(1+\norma{\rho_o^1}_{\L1})}.
  \end{eqnarray}
  For $j = 1, 2$ and $i= 1, \ldots, n$, define $V^i_j =
  v_j^i\left(t,x,\left(\mathcal{A}^i (\rho_j \left(t)\right)\right)
    (x), p_j (t)\right)$ and using~\eqref{eq:DeltaDiv},
  \eqref{eq:quella}, \textbf{(v)}
  and~\textbf{($\boldsymbol{\mathcal{A}}$)}, compute preliminary the
  following terms
  \begin{eqnarray}
    \nonumber
    \norma{\div (V_1^i - V_2^i)}_{\L1 ([0,t];\L1)}
    \!\!\!\! & \leq & \!\!\!\!
    \O
    \left(1+\norma{\rho_o^2}_{\L1}\right)
    \left(
      \norma{\rho_1-\rho_2}_{\L1([0,t],\L1)}
      +
      \norma{p_1-p_1}_{\L1}
    \right)\,,
    \\
    \nonumber
    \norma{\nabla_x \div V_1^i}_{\L1 ([0,t];\L1)}
    \!\!\!\! & \leq & \!\!\!\!
    t \, \O
    \left(
      1      +
      \norma{\rho_o^1}_{\L1}
      +
      \left( \norma{\rho_o^1}_{\L1} \right)^2
    \right) \,,
    \\
    \nonumber
    \norma{V_1^i - V_2^i}_{\L1([0,t];\L\infty)}
    \!\!\!\! & \leq & \!\!\!\!
    \norma{\mathcal{C}_K}_{\L\infty} \!\!\!
    \int_0^t \!\!
    \left[
      \norma{
        \mathcal{A}^i\left(\rho_1 (\tau)\right)
        -
        \mathcal{A}^i\left(\rho_2 (\tau)\right)
      }_{\L\infty} \!\!\!
      + \!
      \norma{p_1 (\tau) - p_2 (\tau)}_{\reali^m} \!
    \right] \!\!
    \d{\tau}
    \\
    \nonumber
    \!\!\!\! & \leq & \!\!\!\!
    \norma{\mathcal{C}_K}_{\L\infty}
    \int_0^t
    \left(
      L_A
      \norma{\rho_1 (\tau) - \rho_2 (\tau)}_{\L1}
      +
      \norma{p_1 (\tau) - p_2 (\tau)}_{\reali^m}
    \right)
    \d{\tau}
    \\
    \label{eq:VmenoVbis}
    \!\!\!\! & = & \!\!\!\!
    \O \left(
      \norma{\rho_1 - \rho_2}_{\L1 ([0,t];\L1)}
      +
      \norma{p_1 - p_2}_{\L1}
    \right) \,.
  \end{eqnarray}
  By Lemma~\ref{lem:stab-1} we get
  \begin{eqnarray*}
    \!\!\!\!
    & &
    \!\!\!\!
    \norma{\rho_1^i(t) - \rho_2^i(t)}_{\L1}
    \\
    \!\!\!\!
    & \leq &
    \!\!\!\!
    \norma{\rho_o^{1,i} - \rho_o^{2,i}}_{\L1} e^{\kappa t}
    +
    \norma{q}_{\L\infty}
    \norma{\div (V_1^i - V_2^i)}_{\L1 ([0,t];\L1)}
    e^{\kappa t}
    \\
    \!\!\!\!
    & &
    \!\!\!\!
    +
    \frac{\kappa_o e^{\kappa_o t} - \kappa e^{\kappa t}}{\kappa_o - \kappa} \!
    \left[
      \tv(\rho_o^{1,i})
      +
      d\, W_d  \norma{q}_{\L\infty} \!
      \norma{\nabla_x \div V_1^i}_{\L1 ([0,t];\L1)}
    \right] \!\!
    \norma{q'}_{\L\infty} \!
    \norma{V_1^i - V_2^i}_{\L1([0,t];\L\infty)}
    \\
    \!\!\!\!
    & = &
    \!\!\!\!
    \left(1+t\O\right)
    \norma{\rho_o^{1,i} - \rho_o^{2,i}}_{\L1}
    \\
    \!\!\!\!
    & &
    \!\!\!\!
    +
    \O
    \left(
      \norma{\div (V_1^i - V_2^i)}_{\L1 ([0,t];\L1)}
      +
      \left[
        1
        +
        \norma{\nabla_x \div V_1^i}_{\L1 ([0,t];\L1)}
      \right]
      \norma{V_1^i - V_2^i}_{\L1([0,t];\L\infty)}
    \right)
    \\
    \!\!\!\!
    & = &
    \!\!\!\!
    \left(1+t\O\right)
    \norma{\rho_o^{1,i} - \rho_o^{2,i}}_{\L1}
    \\
    \!\!\!\!
    & &
    \!\!\!\!
    +
    \O
    \left(
      1      +
      \norma{\rho_o^1}_{\L1}
      +
      \left( \norma{\rho_o^1}_{\L1} \right)^2
    \right)
    \left(
      \norma{\rho_1 - \rho_2}_{\L1 ([0,t];\L1)}
      +
      \norma{p_1 - p_2}_{\L1}
    \right).
  \end{eqnarray*}
  A further application of Lemma~\ref{lem:Gronwall}, using
  also~\eqref{eq:pp}, allows to conclude the proof of~3.~in
  Theorem~\ref{thm:main}.

  \paragraph{Stability with respect to $q$.}
  Fix a positive $T$.  For $j = 1,2$, let $(\rho_j,p_j)$
  solve~\eqref{eq:1} with $q$ replaced by $q_j$ and with the initial
  datum $(\rho_o,p_o)$ assigned at time $t=0$.  For $j = 1, 2$ and $i=
  1, \ldots, n$, define $V^i_j = v^i\left(t,x,\left(\mathcal{A}^i
      (\rho_j \left(t)\right)\right) (x), p_j (t)\right)$.
  Using~\eqref{eq:PrimaDikappat}, \eqref{eq:kappakappa},
  \eqref{eq:VmenoVbis}, \eqref{eq:quella}, \eqref{eq:DeltaDiv},
  \eqref{eq:prima} compute preliminary
  \begin{eqnarray*}
    \kappa_o
    & = &
    \O \,   \max_{j=1,2}\norma{q_j'}_{\L\infty},
    \\
    \kappa
    & = &
    \O \,   \max_{j=1,2}\norma{q_j'}_{\L\infty},
    \\
    \nonumber
    \norma{V^i_1 - V^i_2}_{\L1([0,t];\L\infty)}
    & = &
    \O
    \left(
      \norma{\rho_1 - \rho_2}_{\L1([0,t];\L1)}
      +
      \norma{p_1 - p_2}_{\L1}
    \right),    \\
    \nonumber
    \norma{\nabla_x \div V^i_1(\tau, x)}_{\L1 ([0,t];\L1)}
    & = &
    t \, \O,
    \\
    \nonumber
    \norma{\div \left(V^i_1(\tau,x) - V^i_2(\tau,x)\right)}_{\L1 ([0,t];\L1)}
    & = &
    \O
    \left(
      \norma{\rho_1-\rho_2}_{\L1([0,t];\L1)}
      +
      t \,\norma{p_1 - p_2}_{\C0}
    \right),
    \\
    \nonumber
    \norma{\div V^i_2(\tau,x)}_{\L1 ([0,t];\L1)}
    & = &
    t \, \O \,.
  \end{eqnarray*}
  Apply Lemma~\ref{lem:uniform-p-estimate} to obtain
  \begin{eqnarray*}
    \norma{p_1 (t) - p_2 (t)}_{\reali^m}
    & \leq &
    \O \, \norma{\rho_1 - \rho_2}_{\L1([0,t];\L1)} \,.
  \end{eqnarray*}
  Similarly, using Lemma~\ref{lem:stab-1},
  \begin{eqnarray*}
    & &
    \norma{\rho_1(t) - \rho_2(t)}_{\L1}
    \\
    & = &
    \O
    \left(
      \norma{q_1 - q_2}_{\W1\infty}
      +
      \norma{V_1 - V_2}_{\L1([0,t];\L\infty)}
      +
      \norma{\div (V_1 - V_2)}_{\L1 ([0,t];\L1)}
    \right)
    \\
    & = &
    \O
    \left(
      \norma{q_1 - q_2}_{\W1\infty}
      +
      \norma{\rho_1-\rho_2}_{\L1([0,t];\L1)}
      +
      t \, \norma{p_1 - p_2}_{\C0}
    \right) .
  \end{eqnarray*}
  A final application of Lemma~\ref{lem:Gronwall} completes the proof
  of this part.

  \paragraph{Stability with respect to $v$.}
  Fix a positive $T$. For $j = 1,2$, let $(\rho_j,p_j)$
  solve~\eqref{eq:1} with $v$ replaced by $v_j$ and with the initial
  datum $(\rho_o,p_o)$ assigned at time $t=0$. For $i= 1, \ldots, n$,
  define $V^i_j = v_j^i\left(t,x,\left(\mathcal{A}^i (\rho_j
      \left(t)\right)\right) (x), p_j (t)\right)$. Compute
  preliminary, using~\eqref{eq:kappakappa}, \textbf{(v)},
  \textbf{($\boldsymbol{\mathcal{A}}$)}
  \begin{eqnarray*}
    \kappa_o
    & = &
    \O,
    \\
    \kappa
    & = &
    \O,
    \\
    \nonumber
    \norma{V^i_1 - V^i_2}_{\L1([0,t];\L\infty)}
    & = &
    t \norma{v_1-v_2}_{\L\infty}
    \\
    & &
    +
    \O
    \left(
      \norma{\rho_1 - \rho_2}_{\L1([0,t];\L1)}
      +
      \norma{p_1 - p_2}_{\L1}
    \right),    \\
    \nonumber
    \norma{\nabla_x \div V^i_1(\tau, x)}_{\L1 ([0,t];\L1)}
    & = &
    t \, \O,
    \\
    \nonumber
    \norma{\div \left(V^i_1(\tau,x) - V^i_2(\tau,x)\right)}_{\L1 ([0,t];\L1)}
    & = &
    t \norma{\div (v_1-v_2)}_{\L\infty}
    \\
    & &
    +
    \O
    \left(
      \norma{\rho_1-\rho_2}_{\L1([0,t];\L1)}
      +
      t \,\norma{p_1 - p_2}_{\C0}
    \right),
  \end{eqnarray*}
  so that
  \begin{eqnarray*}
    & &
    \norma{\rho_1(t) - \rho_2(t)}_{\L1}
    \\
    & = &
    \O \frac{\kappa_o e^{\kappa_o t} - \kappa e^{\kappa t}}{\kappa_o - \kappa}
    \left[
      1
      +
      \norma{\nabla_x \div V_1}_{\L1 ([0,t];\L1)}
    \right]
    \norma{V_1 - V_2}_{\L1([0,t];\L\infty)}
    \\
    & &
    +
    \O \norma{\div (V_1 - V_2)}_{\L1 ([0,t];\L1)}
    \\
    & = &
    \O
    \norma{v_1 - v_2}_{\W1\infty}
    +
    \O
    \left(
      \norma{\rho_1-\rho_2}_{\L1 ([0,t];\L1)}
      +
      \norma{p_1 - p_2}_{\L1 ([0,t];\reali^m)}
    \right),
  \end{eqnarray*}
  and similarly to the previous step, applying
  Lemma~\ref{lem:uniform-p-estimate} and Lemma~\ref{lem:stab-1}
  \begin{eqnarray*}
    \norma{p_1 (t) - p_2 (t)}_{\reali^m}
    & \leq &
    \O \norma{\rho_1 - \rho_2}_{\L1 ([0,t];\L1)},
    \\
    \norma{\rho_1(t) - \rho_2(t)}_{\L1}
    & \leq &
    \O
    \norma{v_1 - v_2}_{\W1\infty}
    +
    \O
    \norma{\rho_1-\rho_2}_{\L1 ([0,t];\L1)},
  \end{eqnarray*}
  the proof of the stability with respect to $v$ is completed.

  \paragraph{Stability with respect to $F$.} Apply~\eqref{eq:p} in
  Lemma~\ref{lem:uniform-p-estimate} and Lemma~\ref{lem:stab-1} to
  obtain
  \begin{eqnarray*}
    \norma{p_1 (t) - p_2 (t)}_{\reali^m}
    & \leq &
    \O
    \left(
      t
      \norma{F_1 - F_2}_{\L\infty}
      +
      \norma{\rho_1 - \rho_2}_{\L1 ([0,t];\L1)}
    \right),
    \\
    \norma{\rho_1 (t) - \rho_2 (t)}_{\L1}
    & \leq &
    \O
    \left(\norma{\rho_1 - \rho_2}_{\L1 ([0,t];\L1)}
      +
      \norma{p_1 - p_2}_{\L1 ([0,t];\reali^m)}
    \right),
  \end{eqnarray*}
  and a further application of Lemma~\ref{lem:Gronwall} completes the
  proof.
\end{proofof}

\begin{proofof}{Corollary~\ref{cor:zero}}
  Define $K = \overline{B (\spt \rho_o, \norma{v}_{\L\infty}
    T)}$. Note that for any function $\rho \in \L1 (\reali^d;
  [0,R]^n)$ with $\spt \rho \subseteq K$,
  by~\textbf{($\boldsymbol{\mathcal{A}}$)}
  \begin{displaymath}
    \norma{\mathcal{A} (\rho)}_{\L\infty (\reali^d; \reali^d)}
    \leq
    L_A \norma{\rho}_{\L1 (\reali^d; \reali^n)}
    \leq
    L_A R \, \mathcal{L} ^d(K) \,.
  \end{displaymath}
  By~\eqref{eq:uniform-p-estimate}, for all $t \in [0,T]$, we have $\norma{p
    (t)}_{\reali^m} \leq P$ where
  \begin{displaymath}
    P
    =
    \left(
      \norma{p_o}_{\reali^m}
      +
      \left(1+L_B \, R \, \mathcal{L}^d(K) \right) \norma{C_F}_{\L1 ([0,T]; \reali)}
    \right)
    \exp \norma{C_F}_{\L1 ([0,T]; \reali)} \,.
  \end{displaymath}
  Let $\chi_x \in \Cc\infty (\reali^d; [0,1])$ be such that $\chi_x
  (x) =1$ for all $x \in K$. Similarly, let $\chi_A \in \Cc\infty
  (\reali^d; [0,1])$ such that $\chi_A (A) = 1$ for all $A \in
  \overline{B \left(0, L_A \, R \, \mathcal{L}^d(K)\right)}$ and let
  $\chi_p \in \Cc\infty (\reali^m; [0,1])$ be such that $\chi_p (p) =
  1$ for all $p \in \overline{B (0,P)}$. Then, \textbf{(v.2)} is
  satisfied with
  \begin{eqnarray*}
    \tilde v (t,x,A,p)
    & = &
    \chi_x (x) \, \chi_A (A) \, \chi_p (p) \, v (t,x,A,p),
    \\
    \mathcal{C}_K (x)
    & = &
    \chi_x (x) \,
    \norma{v}_{\C2 ([0,T] \cup K \times \overline{B (0, L_A R
        \mathcal{L}^d (K))} \times \overline{B (0, P)})}
    \,.
  \end{eqnarray*}
  By~1.~in~Definition~\ref{def:solution-1}, the solution $t \to
  \left(\rho (t), p (t)\right)$ to~\eqref{eq:1tilde} as constructed in
  Theorem~\ref{thm:main} is such that $\spt \rho (t) \subseteq K$ for
  all $t \in [0,T]$. Hence, $t \to \left(\rho (t), p (t)\right)$ also
  solves~\eqref{eq:1}.
\end{proofof}

\begin{proofof}{Proposition~\ref{prop:Ex1}}
  Assumption~\textbf{(v.1)} is immediate.
  The verification of~\textbf{(F)} and~\textbf{(q)}, with $R=1$, is
  straightforward.  Assumption~($\boldsymbol{\mathcal{A}}$) directly
  follows from the fact that the map $\nu \colon \reali^2 \to
  \reali^2$ defined by $\nu (x) = \left.x \middle/
    \sqrt{1+\norma{x}^2_{\reali^2}} \right.$ is of class $\C3 (\reali^2;
  \reali^2)$ and $\norma{\nu}_{\C3} < +\infty$. By the standard
  properties of the convolution product, we deduce that
  \begin{eqnarray*}
    & &
    \norma{\mathcal B(\rho_1) - \mathcal B(\rho_2)}_{\W1\infty}
    \\
    & = &
    \norma{\mathcal B(\rho_1) - \mathcal B(\rho_2)}_{\L\infty}
    + \norma{D_p\mathcal B(\rho_1) - D_p\mathcal B(\rho_2)}_{\L\infty}
    \\
    & = &
    \sup_{(p^1, p^2) \in \reali^4} \norma{\left(\rho^1_1 \ast
        \bar \eta(p^1), \rho^2_1 \ast \bar \eta(p^2)\right)}_{\reali^2}
    + \sup_{(p^1, p^2) \in \reali^4} \norma{\left(\rho^1_1 \ast
        D\bar \eta(p^1), \rho^2_1 \ast D\bar \eta(p^2)\right)}_{\reali^4}
    \\
    & \leq & \norma{\bar \eta}_{\W1\infty} \norma{\rho_1 - \rho_2}_{\L1},
  \end{eqnarray*}
  which implies ($\boldsymbol{\mathcal{B}}$), concluding the proof.
\end{proofof}

\begin{proofof}{Proposition~\ref{prop:Ex2}}
  Assumption~\textbf{(v.1)} is immediate.  The verification
  of~\textbf{(F)} and~\textbf{(q)}, with $R=1$, is straightforward.
  Assumption~($\boldsymbol{\mathcal{A}}$) follows in the same way as
  in the proof of Proposition~\ref{prop:Ex1}.  Standard properties of
  the convolution product permit to verify
  assumption~($\boldsymbol{\mathcal{B}}$).
\end{proofof}

\begin{proofof}{Proposition~\ref{prop:Ex3}}
  The proofs of~\textbf{(v.1)}, \textbf{(F)} and~\textbf{(q)} are
  immediate, with $R=1$. To
  prove~\textbf{($\boldsymbol{\mathcal{A}}$)}, note that the real
  valued function $\phi (\xi) = \left. \xi \middle /
    \sqrt{1+\xi^2}\right.$ is globally Lipschitz continuous and the
  map $(r_1,r_2) \to \phi (r_1 \, r_2)$ is Lipschitz continuous for
  $(r_1,r_2) \in [0,1]^2$. The standard properties of the convolution
  also ensure the Lipschitz continuity and the boundedness of the maps
  $\rho \to \eta * (\rho - \bar \rho)$ and $\rho \to \nabla_x
  (\eta*\rho)$ in the required norms,
  proving~\textbf{($\boldsymbol{\mathcal{A}}$)}. The proof
  of~\textbf{($\boldsymbol{\mathcal{B}}$)} is entirely analogous.
\end{proofof}

\medskip

\smallskip\noindent\textbf{Acknowledgment:} This work was partially
supported by the INDAM--GNAMPA project \emph{Conservation Laws: Theory
  and Applications} and the Graduiertenkolleg 1932
  \emph{``Stochastic Models for Innovations in the Engineering Sciences''}.

\small{

  \bibliography{10}

  \bibliographystyle{abbrv} }

\end{document}